\documentclass[11pt]{article}
\usepackage{amsmath}
\usepackage{amssymb}
\usepackage{amsbsy}
\usepackage{amsthm}
\usepackage{epsfig}
\usepackage{wrapfig}
\usepackage{color,cancel}
\usepackage{fullpage}
\usepackage{float}
\usepackage{ulem}
\usepackage{breqn}
\usepackage{graphicx}
\usepackage{caption}
\usepackage{subcaption}
\usepackage{bbm}

\usepackage{verbatim}

\usepackage{url} % traian

\usepackage{amsmath,amsfonts,amssymb,color,graphicx}

\newtheorem{theorem}{Theorem}[section]
\newtheorem{Algorithm}{Algorithm}[section]
\newtheorem{corollary}{Corollary}[section]
\newtheorem{lemma}{Lemma}[section]

\newtheorem{remark}{Remark}[section]

\numberwithin{equation}{section}
\newcommand{\norm}[1]{\left\Vert#1\right\Vert}

\newcommand{\bu}{{\bf u}}
\newcommand{\bv}{{\bf v}}
\newcommand{\bw}{{\bf w}}
\newcommand{\bx}{{\bf x}}
\newcommand{\be}{{\bf e}}
\newcommand{\bff}{{\bf f}}
\newcommand{\bphi}{{\boldsymbol \phi}}
\newcommand{\bfX}{{\bf X}}\newcommand{\bfH}{{\bf H}}
\newcommand{\bfV}{{\bf V}}
\newcommand{\bfL}{{\bf L}}
\newcommand{\bpsi}{{\boldsymbol \psi}}
\newcommand{\bfeta}{{\boldsymbol \eta}}
\newcommand{\btheta}{{\boldsymbol \theta}}

\newcommand*\diff{\mathop{}\!\mathrm{d}}
\newcommand{\beps}{{\boldsymbol  \varepsilon}}
\usepackage{color}
\title{Stokes with variable viscosity}
\usepackage[section]{placeins}
\makeatletter
\AtBeginDocument{%
	\expandafter\renewcommand\expandafter\subsection\expandafter{%
		\expandafter\@fb@secFB\subsection
	}%
}
\date{}

\title{A Modular Regularized Variational Multiscale Proper Orthogonal Decomposition for Incompressible Flows }

\author{
Fatma G. Eroglu
\thanks{Department of Mathematics, Middle East Technical University, 06800 Ankara, Turkey; fguler@metu.edu.tr}
\and
Songul Kaya
\thanks{
Department of Mathematics, Middle East Technical University, 06800 Ankara, Turkey; smerdan@metu.edu.tr}
\and
Leo G. Rebholz
\thanks{Department of Mathematical Sciences, Clemson University, Clemson, SC, 29634; rebholz@clemson.edu; Partially supported by NSF DMS1522191 and U.S. Army Grant 65294-MA.}
}

\begin{document}

\maketitle

\begin{abstract} In this paper, we propose, analyze and test a post-processing implementation of a projection-based variational multiscale  (VMS) method
with proper orthogonal decomposition (POD) for the incompressible Navier-Stokes equations.  The projection-based VMS stabilization is added as a separate post-processing step to the standard POD approximation, and since the stabilization step is completely decoupled, the method can easily be incorporated into existing codes, and stabilization parameters can be tuned independent from the time evolution step.  We present a theoretical analysis of the method, and give results for several numerical tests on benchmark problems which both illustrate the theory and show the proposed method's effectiveness.
\end{abstract}


\section{Introduction}

%\cite{Gue99, Hu95}
%This report considers a post-processing method which was originally developed by Layton et.al. %\cite{LRT11} for the \eqref{nse}. 

We consider the incompressible Navier-Stokes equations (NSE) on a polyhedral domain $\Omega \subset \mathbb R^d, d\in\{2,3\}$ with boundary $\partial\Omega$:
 \begin{equation}\label{nse}
 \begin {array}{rcll}
\bu_t -\nu \Delta \bu+ (\bu\cdot\nabla)\bu + \nabla p &=& \bff  &
 \mathrm{in }\ (0,T]\times \Omega, \\
  \nabla \cdot \bu&=& 0& \mathrm{in }\ [0,T]\times  \Omega,\\
\bu&=& \mathbf{0}& \mathrm{in }\ [0,T]\times\partial \Omega,\\
\bu(0,\bx) & = & \bu_0 & \mathrm{in }\ \Omega, \\
\displaystyle \int_{\Omega}p \ d \bx&=&0,&  \mathrm{in }\ (0,T].
 \end{array}
\end{equation}
Here, $\bu(t,\bx)$ is the fluid velocity and $p(t,\bx)$ the fluid
pressure. The parameters in \eqref{nse} are the kinematic viscosity $\nu >0$,
the prescribed body forces $\bff(t,\bx)$ and the initial velocity field $\bu_0(\bx)$.

It is known that due to the wide range of scales in many complex fluid flows, simulating these flows by a direct numerical simulation (DNS) can be very expensive, and sometimes is even infeasible.  In particular, in the engineering design process, flow simulations must be run many times, e.g.  to perform parameter studies or for system control purposes; this multiplies the DNS cost by at least several times.  The concept of reduced order models was introduced as a way of lowering the computational complexity in such settings.  The proper orthogonal decomposition (POD) approach has proven to be quite successful for generating reduced order models \cite{L87a,L87b,L87c} that capture many (sometimes most) of the dominant flow features, and has been applied in many areas such as image processing, pattern recognition, unsteady aerodynamic applications etc., see e.g. \cite{B93,C00,L03}. POD reduces the complexity of systems, often by orders of magnitude, by representing it with only its most energetic structures.

However, it is well known that for certain problems, and incompressible NSE in particular, a POD by itself can perform quite poorly without some sort of numerical stabilization \cite{ALT93,TZ12}.  One type of stabilization that has been successful in this endeavor is the combining of POD with the variational multiscale (VMS) method.  VMS was introduced in \cite{Hu95} in a variational setting, and it (and its variants) have been studied extensively in finite element frameworks \cite{Gue99,Lay02,jokaya1,jokaya2,jokaya3}.  Using VMS in POD was pioneered in \cite{TZ12,TZ13,JPR13}, and their studies showed this could increased numerical accuracy for convection-dominated convection-diffusion equations \cite{TZ13} and for NSE \cite{TZ12,JPR13}.  Furthermore, in \cite{TZ13}, an analysis was performed that showed 
optimal error bounds could be obtained (in terms of mesh width, time step size, and eigenvalues and eigenvectors removed from the system)
for convection-diffusion-reaction VMS-POD systems.

The main objective of this report is to extend the novel ideas of \cite{LRT11} to the POD setting, in particular to create a VMS-POD where stabilization is added as a completely decoupled second step, in a time stepping scheme for incompressible flow simulation.  That is, at each step there is a two step procedure at each time step: step 1 evolves with a standard POD (i.e. unstabilized Galerkin POD), and then the second step is a weighted POD projection that adds (in a sense) extra viscosity to the lower POD modes (hence the method can be easily incorporated into a standard Galerkin POD code).  We formally introduce the method in section 3 (after giving necessary notation and preliminaries in section 2), and prove it is stable and optimally convergent (in terms of the mesh width, time step size, and removed POD modes/eigenvalues).  In section 4 we provide extensive numerical results that illustrate the effectiveness of the method.  Finally, conclusions are given in section 5.

\section{Preliminaries and Notations}

We assume that $\Omega \subset \mathbb{R}^d,\,d=2,3$ is a polygonal or polyhedral domain, with boundary $\partial \Omega$. Throughout the paper standard notations for Sobolev spaces and their norms will be used. The norm in $(H^k(\Omega))^d$ is denoted by $\|\cdot\|_k$ and the norms in Lebesgue spaces
$(L^p(\Omega))^d$, $1\leq p < \infty$, $p\neq 2$ by $\|\cdot\|_{L^p}$. The space $L^2(\Omega)$ is equipped with the norm and inner product $\|\cdot\|$ and $(\cdot, \cdot)$, respectively. 

The continuous velocity and pressure spaces are denoted by $\bfX= (H^1_0(\Omega))^d $ and $Q=L_0^2(\Omega)$, and we denote the dual space of $H_0^1(\Omega)$ by $H^{-1}$ with norm
\[
\norm{\bff }_{-1}  = \sup_{\bv\in{X}}
\frac{|(\bff,\bv)|}{\norm{\nabla \bv }}.
\] 

We use the following notation for discrete norms, for $\bv^n \in \bfH^p(\Omega)$, n=0,1,2,...,M:
\begin{equation*}
|||\bv|||_{\infty,p}:=\max_{0\leq n\leq M}\|\bv^n\|_p \quad |||\bv|||_{m,p}:=\bigg(\Delta t \sum_{n=0}^M \|\bv^n\|_p^m  \bigg)^{1/m}.
\end{equation*}
%In partial,
%\begin{equation*}
% |||\bv|||_{m,0}:=\bigg(\Delta t \sum_{n=0}^M \|\bv^n\|^m  \bigg)^{1/m}.
%\end{equation*}
%

The variational formulation of (\ref{nse}) reads as follows: Find $\bu \ :\ (0,T]\to
{\bf X}, \ p
\ : \ (0,T]\to Q$ satisfying
\begin{equation}\label{weak_nse_full}
\begin{array}{rcll}
(\bu_t,\bv) + (\nu\nabla{\bu},\nabla{\bv}) + b(\bu,\bu,\bv)
-(p,\nabla\cdot\bv) &= &(\bff,\bv) & \forall \ \bv\in \bfX, \\
(q,\nabla\cdot\bu) & = & 0 & \forall \ q\in Q,
\end{array}
\end{equation}
with $\bu(0,\bx) = \bu_0(\bx) \in \bfX$ and 
$$
b(\bu,\bv,\bw) =
\frac{1}{2}\left(((\bu\cdot\nabla)\bv,\bw)-((\bu\cdot\nabla)\bw,\bv)
\right)
$$
is the skew-symmetric form of the convective term.

The following properties for the skew symmetric form are well known \cite{GR79,WJL8}.
\begin{lemma} \label{l:ssb}
	The trilinear skew-symmetric form $b(\bu,\bv,\bw)$ satisfies  
	\begin{eqnarray*}
	b(\bu,\bv,\bw)&\leq& C(\Omega) \sqrt{\|\bu\|\|\nabla \bu\|}\|\nabla \bv\|\|\nabla \bw\|,\label{ssb}
	\\
	b(\bu,\bv,\bw)&\leq& C (\Omega) \|{\nabla \bu}\|\|{\nabla \bv}\|\|{\nabla \bw}\|.    \label{tri1}  
	\end{eqnarray*}
\end{lemma}
%\begin{proof}
%These estimates  can be derived by applying H\"older's inequality, Sobolev embeddings, \cite{GR79,WJL8}.
%\end{proof} 

This paper considers a conforming finite element method for \eqref{weak_nse_full}, with spaces $\bfX^h\subset \bfX$ and $Q^h \subset 
Q$ satisfying the inf-sup condition: there is a constant
$\beta$ independent of the mesh size $h$ such that
\begin{eqnarray}
\inf_{q_h\in{Q}^h}\sup_{\bv_h\in {\bfX}^h}\frac{(q_h,\,\nabla\cdot
\bv_h)}{||\,\nabla \bv_h\,||\,||\,q_h\,||}\geq \beta > 0.
\label{infsup}
\end{eqnarray}
It will also be assumed for simplicity that the finite element spaces $\bfX^h$ and $Q^h$ are composed of piecewise polynomials of degrees at most $m$ and $m-1$, respectively (the analysis can be extended without significant difficulty to any inf-sup stable pair).
In addition, we assume that the spaces satisfy the following approximation properties 
%{\color{blue} LEO SAYS: I REMOVED DELTA T'S IN BELOW.  THIS IS ABOUT FE SPACE APPROXIMATION PROPERTIES, SO CONFUSED WHY DELTA T'S INVOLVED.  Maybe you are trying to make $v_h$ the FEM approximation of the true solution u?  i.e. like (3.16) in TZ13?}
 \begin{eqnarray}
 \inf_{\bv_h\in{X}^h} \left( \|(\bu-\bv_h)\|+h\|\nabla(\bu-\bv_h)\| \right)&\leq&
C h^{m+1} \| \bu \|_{m+1} \label{ap1}
\\
\inf_{q_h\in{Q}^h}\|p-q_h\|&\leq & Ch^m \|p\|_{m} \label{ap2}
\end{eqnarray}
for $\bu \in \bfX\cap {\bf H}^{m+1}(\Omega)$ and $p\in Q\cap H^{m}(\Omega)$.

We denote the discretely divergence free space by
\begin{equation}
\bfV^h= \lbrace \bv_h \in \bfX^h: (\nabla \cdot \bv_h, q_h)=0, \forall q_h \in Q^h\rbrace,
\end{equation}
The inf-sup condition (\ref{infsup}) implies that the space $\bfV^h$ is closed subspace of $\bfX^h$ and the 
formulation in $\bfX^h$ is equivalent to $\bfV^h$.
%V= \lbrace v \in \bfX: (\nabla \cdot v, q)=0, \forall q \in Q\rbrace,
Thus, the Galerkin finite element approximation of (\ref{weak_nse_full}) in $\bfV^h$ has the  
following form: Find $\bu_h \in \bfV^h$ satisfying
\begin{equation}\label{fe_weak_nse_full}
\begin{array}{rcll}
(\bu_{h,t},\bv_h) + (\nu\nabla{\bu_h},\nabla{\bv_h}) + b(\bu_h,\bu_h,\bv_h) = (\bff,\bv_h), \quad \forall \ \bv_h\in \bfV^h.
\end{array}
\end{equation}

\subsection{POD preliminaries}

%{\color{blue}  The description of POD had all X norms, but we use L2, as does TZ13.  So this is changed.}\\

In this section the essentials of the POD are described. For a more detailed description of the method, see e.g. \cite{H96}.  With the same notation of \cite{TZ12}, consider the finite number of the time instances, 
$$\mathcal{R}= span\{\bu(\cdot,t_1),\dots,\bu(\cdot,t_M)\}$$ 
at time $t_i=i\Delta t$, $i=1,\dots,M$ and let $\Delta t=\frac{T}{M}$, where $rank(\mathcal{R})=d$. In what follows, these time instances will be assumed to come from a DNS computed with a finite element spatial discretization.  That is, we replace $\bu(\cdot,t_k)$ with DNS computed solution $\bu_h(\cdot,t_k)$.

Let $\{\bpsi_1,\bpsi_2,\dots,\bpsi_r\}$ be the low-dimensional ordered POD basis functions to approximate these time instances. The low-dimensional space is obtained by solving the minimization problem of 
\begin{equation}\label{minpod}
\min \frac{1}{M}\sum_{k=1}^{M}\|\bu(\cdot,t_k)-\sum_{i=1}^{r}(\bu(\cdot,t_k),\bpsi_i(\cdot)) \bpsi_i(\cdot)\|^2,
\end{equation}
such that $(\bpsi_i,\bpsi_j)=\delta_{ij}$, $1\leq i,j \leq r$ and $r<<d$.
The solution of the problem (\ref{minpod}) yields
\begin{eqnarray}\label{solnpod}
\bpsi_l(\cdot)
=\frac{1}{\sqrt{\lambda_l}}\sum_{i=1}^{M}(\bv_l)_i \bu(\cdot,t_i), \quad 1\leq l \leq r, 
\end{eqnarray}
where $(\bv_l)_i$ is the $i^{th}$ component of the eigenvector $\bv_l$ corresponding to $\lambda_l$ which is the eigenvalue of the snapshots correlation matrix. We note that all eigenvalues are sorted in descending orders. Thus, the basis functions $\{\bpsi_1,\bpsi_2,\dots,\bpsi_r\}$ correspond to the first $r$ largest eigenvalues. 
Moreover, the error estimation satisfies (see \cite{KV01}):
\begin{eqnarray}\label{slam}
\frac{1}{M}\sum_{k=1}^{M}\|\bu(\cdot,t_k)-\sum_{i=1}^{r}(\bu(\cdot,t_k),\bpsi_i(\cdot)) \bpsi_i(\cdot)\|^2=\sum_{i=r+1}^{d} \lambda_i.
\end{eqnarray}   
Let $\bfX^r = span\{ \bpsi_1,\bpsi_2,\dots,\bpsi_r\}$ be the POD-space, then POD-Galerkin (POD-G) formulation of the NSE is : Find $\bu_r\in\bfX^r$ satisfying
\begin{equation}\label{fe_proper}
\begin{array}{rcll}
(\bu_{r,t},\bpsi) + (\nu\nabla{\bu_r},\nabla\bpsi) + b(\bu_r,\bu_r,\bpsi) = (\bff,\bpsi), \quad \forall \ \bpsi\in \bfX^r.
\end{array}
\end{equation}
Note that the POD-G solution of the NSE is constructed  by writing, $$\bu_r(\bx,t):=\sum_{i=1}^{r} a_j(t)\bpsi_j(\bx),$$ 
where $a_j(t)$ are time varying coefficients. 

To carry out the error analysis, we state the following error estimations without their proofs. Let $M_r$ and $S_r$ denote the POD mass matrix and stiffness matrix, respectively, with
\begin{eqnarray*}
(M_r)_{i,j}&=&\int_{\Omega}\bpsi_j\bpsi_i \diff \bx
\\
(S_r)_{i,j}&=& \int_{\Omega}\nabla \bpsi_j\cdot \nabla \bpsi_j \diff \bx
\end{eqnarray*}  
 Note that, since we use $L^2(\Omega)$ to generate snapshots, $M_r=I_{r\times r}$.

Our analysis will utilize the following POD inequalities, proven in \cite{KV01}: for all $\bu_r \in \bfX^r$,
\begin{eqnarray}
%\|\bu_r\|&\leq& \|S_r^{-1}\|_2^{\frac{1}{2}}\|\nabla\bu_r\|, \mbox{\color{blue} Why use this instead of Poincare? Is this used?}
\|\nabla \bu_r\|&\leq& \|S_r\|_2^{\frac{1}{2}}\|\bu_r\|,
\end{eqnarray}\label{gradsr}
where $\|\cdot\|_2$ denotes the matrix $2$-norm.

The results of the following lemmas will be applied to bound the POD projection error. The proofs can be found in \cite{TZ12}.

\begin{lemma} 
	The error in POD projection for the snapshots $\bu_h(\cdot,t_k)$, $k=1,\dots, M$ satisfies 
%	{\color{blue} is this for any $u_h \in V_h$?  In (2.9), the L2 equivalent of this, the u's don't have sub-h in them.  Something not right.}
	\begin{eqnarray}
	\frac{1}{M}\sum_{k=1}^{M} \|\bu_h(\cdot,t_k)-\sum_{m=1}^{r}(\bu_h(\cdot,t_k),\bpsi_m(\cdot))\bpsi_m(\cdot)\|_1^2=\sum_{m=r+1}^{d}\|\bpsi_m\|^2_1 \lambda_m. \label{fb}
	\end{eqnarray}
\end{lemma}
The proof of the finite element error estimate consists of the splitting the error into an approximation term and a finite element remainder term. To decompose the error term we use the $L^2$ projection of $\bu$, which fulfills certain interpolation estimates. Lemma \ref{errpro} estimates the error between the snapshots and their $L^2$ projection into $\bfX^r$. 

Let $P_r$ denote a projection operator $P_r: \bfL^2\rightarrow \bfX^r$ that satisfies 
\begin{eqnarray}
(\bu-P_r\bu,\bpsi_r)=0,\quad \forall \bpsi_r\in\bfX^r. \label{l2pro}
\end{eqnarray}

%{\color{blue} Leo says: Be careful here.  I think this $u^n$ has to be a snapshot, or true solution at snapshot time n (either would be ok, I think).  Also, there are regularity assumptions that go here, since there is FEM error estimates involved.  These should be stated, and further, should be stated in convergence theorem. (you can't get an $h^{2m}$ without a $u\in H^{m+1}$...)}\\
%{\color{red} Fatma says: In the following lemma, we include $\|\cdot\|_{m+1}$ term. We state the proof of the lemma.}

\begin{lemma} \label{errpro} 
For $\bu^n$ the true NSE solution at time $t^n$, the difference $\bu^n-P_r\bu^n$ satisfies 
\begin{eqnarray}
\frac{1}{M}\sum_{n=1}^{M}\|\bu^n-P_r\bu^n\|^2&\leq& C\Big(h^{2m+2}\frac{1}{M}\sum_{n=1}^{M}\|\bu^n\|^2_{m+1}+\sum_{i=r+1}^{d}\lambda_i\Big),
	\\
	\frac{1}{M}\sum_{n=1}^{M}\|\nabla (\bu^n-P_r\bu^n)\|^2&\leq& C\Big((h^{2m}+\|S_r\|_2h^{2m+2})\frac{1}{M}\sum_{n=1}^{M}\|\bu^n\|^2_{m+1}+\sum_{i=r+1}^{d}\|\bpsi_i\|_1^2\lambda_i\Big).
	\end{eqnarray}
\end{lemma}
\begin{proof}
This follows from the projection error estimate from \cite{TZ12} along with the triangle inequality, after adding and subtracting the FEM-DNS snapshots to the true solution (at corresponding times).
\end{proof}

{\bf Assumption:} We make the assumption that 
\begin{eqnarray}
\|\bu^n-P_r\bu^n\|^2&\leq& C\Big(h^{2m+2}+\Delta t^2+\sum_{i=r+1}^{d}\lambda_i\Big) \label{assm1} \quad \text{and}
\\
\|\nabla (\bu^n-P_r\bu^n)\|^2&\leq& C\Big(h^{2m}+\|S_r\|_2h^{2m+2}+(1+\|S_r\|_2)\Delta t^2 +\sum_{i=r+1}^{d}\|\bpsi_i\|_1^2\lambda_i\Big).\label{assm2}
\end{eqnarray}
This assumption essentially says that the projection error the true solution is of the same order of magnitude at most of the time steps.  It is possible the assumption may not hold (in which case the $C$ on the right hand side with become $C\Delta t^{-1}$) in diabolical cases where the true NSE solution lives in the ROM space at most of the times $t^n$.  The assumption is common in error analysis for POD type methods  \cite{TZ12, XWWI17}.

%
%{\color{red} Fatma says: I have tried several estimation to remove this assumption. In the estimation of (\ref{wrb}) without this assumption, I have the bounds like
%	\begin{itemize}
%		\item $C {\nu}^{-2} (\Delta t)^{-1}(\|\bu_r^0\|^2+\nu^{-1}|||\bff|||^2_{2,-1})  \Big((h^{2m}+\|S_r\|_2h^{2m+2})|||\bu|||^2_{2,m+1}+\sum_{j=r+1}^{d}\|\bpsi_j\|_1^2\lambda_j\Big)$
%		\item ${\nu}^{-\frac{3}{2}}({\Delta t})^{-\frac{1}{2}}(\|\bu_r^0\|^2+\nu^{-1}|||\bff|||^2_{2,-1})^{3/2} \Big((h^{2m}+\|S_r\|_2h^{2m+2})|||\bu|||^2_{2,m+1}+\sum_{j=r+1}^{d}\|\bpsi_j\|_1^2\lambda_j\Big)$
%	\end{itemize}
%	That's why I believe we need this assumption. 
%	
%	Similarly, I have checked  the paper Giere, Ilescu, John, Wells. It turns out that it is commonly used assumption (see page 17, Hypothesis 3.1 )} 
%	
%\subsection{Projection-based VMS formulation for POD}
%
%The basic idea in VMS is to add the artificial viscosity only to the smallest resolved scales instead of all resolved scales. In this way small scale oscillations are reduced (or hopefully eliminated). In this respect, one can formulate the POD setting in the VMS framework by choosing the appropriate finite element spaces and adding to projection in the formulation of POD-G, (see \cite{TZ12,TZ13} for details). The method that we propose in the next section adds one uncoupled modular projection step for the VMS eddy viscosity.   

In the projection-based VMS method, besides the standard finite element spaces representing all resolved scales an additional large resolved scale is needed. 
For VMS-POD setting, the following spaces are used for $R<r$:
\begin{eqnarray}
 \bfX^R&=&span\{\bpsi_1,\bpsi_2,\dots, \bpsi_R\}, \label{l1}
\\
\bfL^R&=& \nabla \bfX^R :=span\{\nabla\bpsi_1,\nabla\bpsi_2,\dots, \nabla \bpsi_R\} \label{l2}.
\end{eqnarray}
Note that from the construction, we have $\bfX^R \subset \bfX^r \subset \bfX^h \subset \bfX$.

%Since POD basis functions are sorted in descending order with respect to their kinetic energy, we can consider the space $\bfX^R$ as large resolved scale i.e. basis functions corresponding to low energy and the space  $\{\bpsi_{R+1},\bpsi_{R+2},\dots,\bpsi_r\}$ as small resolved scale, with $R<r$. 

The  $L^2$ orthogonal projection $P_R:\bfL^2\rightarrow \bfL^R$ will be needed in the VMS formulation, and is defined by
\begin{eqnarray}\label{ortp}
(\bu-P_R\bu,\bv_R)=0,\quad \forall \bv_R \in \bfL^R.
\end{eqnarray}
%and the initial condition $\bu_r(\cdot,0)$ is given by orthogonal projection of $\bu^0$ on $X^r$,
%\begin{eqnarray}\label{init}
%\bu_r(\cdot,0)=\bu_r^0:=\sum_{i=1}^{r} (\bu^0,\bpsi_j)\bpsi_j(\bx) .
%\end{eqnarray}
%From the definition of the projection, it is clear that 
%\begin{equation}\label{ortp}
%\|\bu-P_R\bu\|\leq CH^k\|\bu\|_k
%\end{equation}

\section{Post-Processed VMS-POD Schemes}

This section proposes fully discrete VMS-POD methods for solving (\ref{fe_proper}).  For simplicity, in section 3.1, we analyze the backward Euler temporal discretization.  In section 3.2, extension to BDF2 time stepping is considered.  In our analysis, we assume that the  eddy viscosity coefficient $\nu_T $ is known bounded, positive and element-wise constant. The results of this report can be extended in the case $\nu_T$ is nonconstant even nonlinear. The consideration of a nonlinear $\nu_T$ requires more complex mathematical theory due to the strong monocity, see \cite{SWZ13}. %{\color{blue}  $\nu_T$ is pulled out of norms/integrals in the analysis, so we treat it as a constant globally.  One could do element-wise, and use its min.  But that would be a pain.  Maybe we could just mention that this can be done?}

In that analysis that follows, we denote variables at time $t^n=n\Delta t, n=0,1,2,\dots,M, T:=M\Delta t$ using superscripts, e.g. ${\bf f}^n:={\bf f}(t^n)$.

\subsection{Backward Euler}

The two step VMS-POD scheme equipped with backward Euler time stepping reads as follows:

\begin{Algorithm} \label{alg1}
Let $\bff \in L^2(0,T;\bfH^{-1}(\Omega))$ and $\bu_r^0=\bw_r^0$  be given with $L^2$ projection of $\bu_0$ in $\bfX^r$. Given $\bu_r^n \in \bfX_r$ compute $\bu_r^{n+1}$ by applying the following two steps:

{\bf Step 1.}
Calculate  $\bw_r^{n+1}\in \bfX^r$ satisfying $\forall {\bpsi\in \bfX^r} $,
\begin{eqnarray} \label{step1}
\qquad\bigg( \frac{\bw_r^{n+1}-\bu_r^{n}}{\Delta t},\bpsi \bigg)+b(\bw_r^{n+1},\bw_r^{n+1},\bpsi)
+ \nu(\nabla \bw_r^{n+1},\nabla \bpsi)=(\bff^{n+1},\bpsi).
\end{eqnarray}

{\bf Step 2.} Post-process $\bw_r^{n+1}$ by applying projection $P_R$ to obtain
$\bu_r^{n+1} \in \bfX_r$, $\forall {\bpsi\in \bfX^r}$:
\begin{eqnarray}\label{step2}
\qquad\bigg (\frac{\bw_r^{n+1}-\bu_r^{n+1}}{\Delta t},\bpsi\bigg ) =(\nu_T(I-P_{R})\nabla \frac{(\bw_r^{n+1}+\bu_r^{n+1})}{2} , (I-P_{R})\nabla \bpsi),
\end{eqnarray}
\end{Algorithm}

We note that Step 1 is the standard Galerkin POD method, and step 2 is completely decoupled VMS stabilization step.  The projection in Step 2 is not a filter but constructed to recover VMS eddy viscosity term as in \cite{LRT11}.  

Note that if we let 
$\bpsi= \frac{(\bw_r^{n+1}+\bu_r^{n+1})}{2}$ in (\ref{step2}), the numerical dissipation induced from Step 2 is immediately seen to be
\begin{eqnarray}
\|{\bw_r^{n+1}\|^2= \|\bu_r^{n+1}}\|^2+2\nu_T\Delta t\norm{(I-P_{R})\nabla \frac{(\bw_r^{n+1}+\bu_r^{n+1})}{2}}^2. \label{numdis}
\end{eqnarray}

\subsubsection{Stability of Algorithm \ref{alg1}}

We now prove stability of Algorithm \ref{alg1}.  

\begin{lemma} \label{Lem:sta} The post-processed-VMS-POD approximation is unconditionally stable in the following sense: for any $\Delta t>0$,
\begin{eqnarray}
\|\bu_r^{M}\|^2+\sum_{n=0}^{M-1}\bigg[2\nu_T\Delta t\norm{\|(I-P_{R})\nabla \frac{(\bw_r^{n+1}+\bu_r^{n+1})}{2}}^2+\|\bw_r^{n+1}-\bu_r^{n}\|^2\nonumber
\\
+\nu\Delta t \|\nabla \bw_r^{n+1}\|^2\bigg]\leq \|\bu_r^{0}\|^2+\nu^{-1}|||\bff|||^2_{2,-1}. \nonumber
\end{eqnarray}

\end{lemma}
\begin{proof}
Letting $\bpsi=\bw^{n+1}_{r}$ in (\ref{step1}) and using the polarization identity  yields
\begin{eqnarray}
\frac{1}{2 \Delta t}\|\bw_r^{n+1}\|^2-\frac{1}{2 \Delta t}\|\bu_r^{n}\|^2+\frac{1}{2 \Delta t}\|\bw_r^{n+1}-\bu_r^{n}\|^2
+\nu \|\nabla \bw_r^{n+1}\|^2=(\bff^{n+1},\bw_r^{n+1}). \label{stab1}
\end{eqnarray}
Substitute (\ref{numdis}) in (\ref{stab1}) and multiply both sides by $2\Delta t$, which provides
\begin{eqnarray}
\|\bu_r^{n+1}\|^2-\|\bu_r^{n}\|^2+2\nu_T\Delta t\bigg\|(I-P_{R})\nabla \frac{(\bw_r^{n+1}+\bu_r^{n+1})}{2}\bigg\|^2+\|\bw_r^{n+1}-\bu_r^{n}\|^2 \nonumber
\\
+2\nu\Delta t \|\nabla \bw_r^{n+1}\|^2=2\Delta t(\bff^{n+1},\bw_r^{n+1}). \label{stab2}
\end{eqnarray}
Bounding the forcing term in the usual way, and then summing over the time steps gives the stated result.
\end{proof}
The result of Lemma \ref{Lem:sta} also establishes the stability of $\bw_r^{M}$.
\begin{corollary} \label{staw} (Stability of  $\bw_r^M$)
\begin{gather}
\|\bw_r^{M}\|^2+2\nu_T\Delta t\sum_{n=0}^{M-2} \bigg\|(I-P_{R})\nabla \frac{(\bw_r^{n+1}+\bu_r^{n+1})}{2}\bigg\|^2+\sum_{n=0}^{M-1}
\bigg[\|\bw_r^{n+1}-\bu_r^{n}\|^2\nonumber
\\
+\nu\Delta t \|\nabla \bw_r^{n+1}\|^2\bigg]\leq \|\bu_r^{0}\|^2+\nu^{-1}|||\bff|||^2_{2,-1} 
\end{gather}
\end{corollary}
\begin{proof}
Expand the summation in Lemma \ref{Lem:sta} for $n=M-1$ and use (\ref{numdis}).
\end{proof}

\subsubsection{A Priori Error Estimation } In this section, we present the error analysis of the true solution of Navier Stokes equations and POD approximation (\ref{step1})-(\ref{step2}).
The optimal asymptotic error estimation requires the following regularity assumptions for the true solution:
%\textcolor{red}{Songul says: Could you please look at the resularity assumptions, I have checked it is OK but for $u_{tt}$, do wee need $L^\infty??$ } 
%{\color{blue}  I think L2 for utt is fine.  But we need much more, unfortunately, since we use the ROM projection error / asssumption, so we will need $u\in L^{\infty}(0,T;H^{m+1}(\Omega))$ and $p\in L^{\infty}(0,T;H^m(\Omega))$.}\\
%
% $$\bu \in L^2(0,T;H^{m+1}(\Omega))\cap L^{\infty}(0,T;H^1(\Omega))$$ 
\begin{eqnarray}
&\bu\in L^{\infty}(0,T;H^{m+1}(\Omega))  
\quad p\in L^{\infty}(0,T;H^{m}(\Omega))\quad \bu_{tt} \in L^2(0,T;H^1(\Omega))\quad \nonumber\\
&\bff \in L^2(0,T;H^{-1}(\Omega)) \label{regularity}
\end{eqnarray} 
\begin{theorem}\label{thm} Suppose (\ref{regularity}) holds and $\bu_r^n$ and $\bw_r^n$ given by Algorithm (\ref{step1})-(\ref{step2}). For sufficiently small $\Delta t$, i.e. $\Delta t \leq [C\nu^{-3}\|\nabla \bu\|^4_{\infty,0}]^{-1}$ we have the following estimation:
\begin{eqnarray}\label{conv:be}
\lefteqn{\|\bu^M-\bu_r^{M}\|^2+ \sum_{n=0}^{M-1}\bigg[\frac{1}{4}\Delta t\nu_T\|(I-P_{R})\nabla (\bu^{n+1}-(\bu_r^{n+1}+\bw_r^{n+1})/2)\|^2}\nonumber
\\ 
&&+ \nu\Delta t \|\nabla (\bu^{n+1}-\bw_r^{n+1})\|^2\bigg] \leq
 C\bigg[h^{2m+2}|||\bu|||^2_{2,m+1}+\sum_{j=r+1}^{d}\lambda_j \nonumber
 \\&&+
\nu \bigg(\Big(h^{2m}+\|S_r\|_2h^{2m+2}\Big)|||\bu|||^2_{2,m+1}+\sum_{j=r+1}^{d}\|\bpsi_j\|_1^2\lambda_j\bigg)\nonumber
\\
&&+ {\nu}^{-1}|||\nabla \bu|||_{2,0}^2\bigg(\Big(h^{2m}+\|S_r\|_2h^{2m+2}\Big)|||\bu|||^2_{2,m+1}+\sum_{j=r+1}^{d}\|\bpsi_j\|_1^2\lambda_j\bigg)\nonumber
\\
&&+\nu_T \bigg(\Big(h^{2m}+(\|S_R\|_2+\|S_r\|_2)h^{2m+2}\Big)|||\bu|||^2_{2,m+1}\nonumber
\\&&+\sum_{j=R+1}^{d}\|\bpsi_j\|_1^2\lambda_j+\sum_{j=r+1}^{d}\|\bpsi_j\|_1^2\lambda_j\bigg)+ {\nu}^{-2} (\|\bu_r^0\|^2+\nu^{-1}|||\bff|||^2_{2,-1})  \nonumber
\\
&&\times\bigg(\Big(h^{2m}+\|S_r\|_2h^{2m+2}\Big)|||\bu|||^2_{2,m+1}+\sum_{j=r+1}^{d}\|\bpsi_j\|_1^2\lambda_j\bigg)
\nonumber
\\
&&+{\nu}^{-1} h^{2m}|||p|||_{2,m}^2+\nu^{-1}(\Delta t)^2 \|\bu_{tt}\|_{L^2(0,T;H^1(\Omega))}^2\bigg]\nonumber
\end{eqnarray}
where $C$ is independent from $\Delta t, h, \nu$ and $\nu_T$.
\end{theorem}
\begin{remark}
Under the assumptions of Theorem \ref{thm} and the finite element spaces $(\bfX^h,Q^h)$ with piecewise polynomials of degree $m$ and $m-1$, respectively. We obtain the following asymptotic error estimation:
\begin{eqnarray}\label{conv:remark}
\lefteqn{\|\bu^M-\bu_r^{M}\|^2+ \sum_{n=0}^{M-1}\bigg[\frac{1}{4}\Delta t\nu_T\|(I-P_{R})\nabla (\bu^{n+1}-(\bu_r^{n+1}+\bw_r^{n+1})/2)\|^2}\nonumber
\\ 
&&+ \nu\Delta t \|\nabla (\bu^{n+1}-\bw_r^{n+1})\|^2\bigg] \leq
C\Big( h^{2m}+(\Delta t)^2+(1+\|S_R\|_2+\|S_r\|_2)h^{2m+2}\nonumber
\\&&+\sum_{j=R+1}^{d}\|\bpsi_j\|_1^2\lambda_j+\sum_{j=r+1}^{d}(1+\|\bpsi_j\|_1^2)\lambda_j \Big)\nonumber
\end{eqnarray}
\end{remark}
%\begin{theorem} \label{conv:be}
%Suppose (\ref{regularity}) holds and $\bu_r^n$ and $\bw_r^n$ given by Algorithm 1. For sufficiently small $\Delta t$ we have the following estimation:
%	\begin{eqnarray*}
%&&\|\bu^M-\bu_r^{M}\|^2+ \sum_{n=0}^{M-1}\bigg[\frac{1}{2}\Delta t\nu_T\|(I-P_{R})\nabla (\bu^{n+1}-(\bu_r^{n+1}+\bw_r^{n+1})/2)\|^2\nonumber
%	\\ &&+ \nu\Delta t \|\nabla (\bu^{n+1}-\bw_r^{n+1})\|^2\bigg] 
%	\leq C\bigg[
%	\Big(\nu +\nu_T+\Delta t \nu^{-1}|||\nabla \bu|||_{2,0}^2\nonumber
%	\\
%	&&+\Delta t \nu^{-1} \big(\nu^{-1/2}|||\bff|||_{2,-1}+\|\bu_r^0\|\big)  |||\nabla \bw|||_{1,0} \Big)\Big(h^{2m}+\|S_r\|_2h^{2m+2}
%	\nonumber
%	\\
%&&+(1+\|S_r\|_2)\Delta t^2+\sum_{j=r+1}^{d}\|\bpsi_j\|_1^2\lambda_j\Big)+\nu_T \Big(h^{2m}+\|S_R\|_2h^{2m+2}\nonumber
%	\\
%&&+(1+\|S_R\|_2)\Delta t^2+\sum_{j=R+1}^{d}\|\bpsi_j\|_1^2\lambda_j\Big)
%+{\nu}^{-1} h^{2m}|||p|||_{2,m+1}^2+\nu^{-1}(\Delta t)^2 \|\bu_{tt}\|_{2,2}^2\bigg]
%	\end{eqnarray*}
%	where  $C$ is independent from $\Delta t, h, \nu$ and $\nu_T$
%\end{theorem}

%{\color{blue} Can we add a remark that reduces this to something reasonable / readable, i.e. a bound just in terms of the parameters (Call everything else a constant)?}\\

\begin{proof}We begin the proof by deriving error equations. From (\ref{weak_nse_full}), we have that true solution $(\bu,p)$ at time level $t=t^{n+1}$
satisfies, for $\bpsi_r\in X_r$,
\begin{gather}
\bigg( \frac{\bu^{n+1}-\bu^{n}}{\Delta t},\bpsi_r \bigg)+\nu(\nabla \bu^{n+1},\nabla \bpsi_r)+b(\bu^{n+1},\bu^{n+1},\bpsi_r) \nonumber
\\
-(p^{n+1},\nabla\cdot\bpsi_r)+E(\bu,\bpsi_r)=(\bff(t^{n+1}),\bpsi_r), \label{err1}
\end{gather}
where
\begin{equation*}
E(\bu,\bpsi)= \bigg( \bu_t^{n+1}-\frac{\bu^{n+1}-\bu^{n}}{\Delta t},\bpsi \bigg).
\end{equation*}
We define the following notations:
\begin{gather}
\bfeta^n:=\bu^n-{\bf \mathcal{U}}^n,\quad \bphi_r^n:=\bw_r^n-{\bf\mathcal{U}}^n,\quad \btheta_r^n:=
\bu_r^n-{\bf\mathcal{U}}^n, \quad
\be_r^n=\bu^n-\bu_r^n,\quad \beps_r^n:=\bu^n-\bw_r^n, \nonumber
\end{gather}
where ${\bf \mathcal{U}^n}$ is $L^2$ projection of $\bu^n$ in $\bfX_r$. 

Subtracting (\ref{step1}) from  (\ref{err1}) yields
\begin{eqnarray}
\bigg( \frac{\beps_r^{n+1}-\be_r^{n}}{\Delta t},\bpsi_r \bigg)+\nu(\nabla \beps_r^{n+1},\nabla \bpsi_r)+b(\bu^{n+1},\bu^{n+1},\bpsi_r)\nonumber
\\
-b(\bw_r^{n+1},\bw_r^{n+1},\bpsi_r)-(p^{n+1},\nabla\cdot\bpsi_r)+E(\bu,\bpsi_r)=0.
\end{eqnarray}
Substitute $\beps_r^n=\bfeta^n-\bphi_r^n$ and $\be_r^n=\bfeta^n-\btheta_r^n$ with $\bpsi=\bphi_r^{n+1}$ in the last equation we obtain
\begin{gather}
\bigg( \frac{\bphi_r^{n+1}-\btheta_r^{n}}{\Delta t},\bphi_r^{n+1}\bigg)+\nu\|\nabla \bphi_r^{n+1}\|^2=\bigg( \frac{\bfeta^{n+1}-\bfeta^{n}}{\Delta t},\bphi_r^{n+1}\bigg) \nonumber
+\nu(\nabla \bfeta^{n+1},\nabla \bphi_r^{n+1}) \nonumber
\\
+[b(\bu^{n+1},\bu^{n+1},\bphi_r^{n+1})-b(\bw_r^{n+1},\bw_r^{n+1},\bphi_r^{n+1})]
-(p^{n+1},\nabla \cdot \bphi_r^{n+1})+E(\bu,\bphi_r^{n+1}). \nonumber
\end{gather}
From the definition of $L^2$ projection (\ref{l2pro}), we note that $(\bfeta^n,\bphi_r^{n+1})=0$ and $(\bfeta^{n+1},\bphi_r^{n+1})=0$.  Using this along with the polarization identity and that $\bphi_r^{n+1} \in \bfX^r \subset \bfV^h$, we obtain the bound
\begin{eqnarray}
\lefteqn{\frac{1}{2\Delta t}(\|\bphi_r^{n+1}\|^2-\|\btheta_r^{n}\|^2)+\nu\|\nabla\bphi_r^{n+1}\|^2 \leq %\bigg |\bigg( \frac{\bfeta^{n+1}-\bfeta^{n}}{\Delta t},\bphi_r^{n+1}\bigg) \bigg | \nonumber
|\nu(\nabla \bfeta^{n+1},\nabla\bphi_r^{n+1})|} \nonumber\\
&&+|b(\bu^{n+1},\bu^{n+1},\bphi_r^{n+1})-b(\bw_r^{n+1},\bw_r^{n+1},\bphi_r^{n+1})| \nonumber
\\
&&+|(p^{n+1}-q_h,\nabla \cdot \bphi_r^{n+1})|+|E(\bu,\bphi_r^{n+1})| \label{boundterm}
\end{eqnarray}
%The first term on the right hand side of (\ref{boundterm}) can be bounded with
%\begin{gather}
%\bigg |\bigg( \frac{\bfeta^{n+1}-\bfeta^{n}}{\Delta t},\bphi_r^{n+1}\bigg) \bigg | \leq \frac{C}{\Delta t \nu} \int_{t^n}^{t^{n+1}} \|\partial_t\bfeta\|^2 d\tau+\frac{\nu}{14}\|\nabla\bphi_r^{n+1}\|^2. \label{b1}
%\end{gather}
The first term on the right hand side of (\ref{boundterm}), and the pressure term, can be bounded using Cauchy-Schwarz and Young's inequalities,
\begin{eqnarray}
|\nu(\nabla \bfeta^{n+1},\nabla \bphi_r^{n+1})| &\leq & C\nu \|\nabla \bfeta^{n+1}\|^2 +\frac{\nu}{12}\|\nabla\bphi_r^{n+1}\|^2 \label{b2} \\
|(p^{n+1}-q_h,\nabla \cdot \bphi_r^{n+1})| &\leq&  \frac{C}{\nu}\|p^{n+1}-q_h\|^2+\frac{\nu}{12}\|\nabla\bphi_r^{n+1}\|^2.
\end{eqnarray}
For the nonlinear terms, first add and subtract terms to get
\begin{eqnarray}
\lefteqn{b(\bu^{n+1},\bu^{n+1},\bphi_r^{n+1})-b(\bw_r^{n+1},\bw_r^{n+1},\bphi_r^{n+1})} \nonumber
\\&=& b(\bu^{n+1},\bu^{n+1},\bphi_r^{n+1}) 
-b(\bw_r^{n+1},\bu^{n+1},\bphi_r^{n+1})+b(\bw_r^{n+1},\bu^{n+1},\bphi_r^{n+1}) -b(\bw_r^{n+1},\bw_r^{n+1},\bphi_r^{n+1}) \nonumber
\\
&=&b(\beps_r^{n+1},\bu^{n+1}, \bphi_r^{n+1})+b(\bw_r^{n+1},\beps_r^{n+1}, \bphi_r^{n+1}) \nonumber
\\
&=&b(\bfeta^{n+1},\bu^{n+1}, \bphi_r^{n+1})-b(\bphi_r^{n+1},\bu^{n+1}, \bphi_r^{n+1}) +b(\bw_r^{n+1},\bfeta^{n+1}, \bphi_r^{n+1}). \label{nonline}
\end{eqnarray}
Now using Lemma \ref{l:ssb},  Young's and Poincar\'e's inequalities, the terms in (\ref{nonline}) are estimated as follows:
\begin{eqnarray*}
|b(\bfeta^{n+1},\bu^{n+1}, \bphi_r^{n+1})|&\leq&C\sqrt{\|\bfeta^{n+1}\|\|\nabla \bfeta^{n+1}\|}\|\nabla\bu^{n+1}\|\|\nabla \bphi_r^{n+1}\|
\\
&\leq& \frac{C}{\nu}\|\bfeta^{n+1}\|\|\nabla \bfeta^{n+1}\|\|\nabla \bu^{n+1}\|^2+\frac{\nu}{12}\|\nabla\bphi_r^{n+1}\|^2,
\\
|b(\bphi_r^{n+1},\bu^{n+1}, \bphi_r^{n+1})|&\leq& C\sqrt{\|\bphi_r^{n+1}\|\|\nabla \bphi_r^{n+1}\|}\|\nabla\bu_r^{n+1}\|\|\nabla \bphi_r^{n+1}\|
\\
&\leq& \frac{C}{\nu^3}\|\bphi_r^{n+1}\|^2\|\nabla\bu^{n+1}\|^4+\frac{\nu}{12}\|\nabla\bphi_r^{n+1}\|^2,
\\
|b(\bw_r^{n+1},\bfeta^{n+1}, \bphi_r^{n+1})|&\leq&C\sqrt{\|\bw_r^{n+1}\|\|\nabla \bw_r^{n+1}\|}\|\nabla \bfeta^{n+1}\|\|\nabla \bphi_r^{n+1}\|
\\
&\leq& \frac{C}{\nu}\|\bw_r^{n+1}\|\|\nabla \bw_r^{n+1}\|\|\nabla\bfeta^{n+1}\|^2+\frac{\nu}{12}\|\nabla\bphi_r^{n+1}\|^2.
\end{eqnarray*}
The consistency error in (\ref{boundterm}) is estimated by
\begin{equation}
|E(\bu,\bphi_r^{n+1})|\leq \frac{C}{\nu}\| \bu_t^{n+1}-\frac{\bu^{n+1}-\bu^{n}}{\Delta t} \|^2+\frac{\nu}{12}\|\nabla\bphi_r^{n+1}\|^2
\end{equation}
Collecting all bounds for the right hand side terms of (\ref{boundterm}) and multiplying both sides by $2 \Delta t$ gives
\begin{eqnarray}
&& \hspace{-.5in} {(\|\bphi_r^{n+1}\|^2-\|\btheta_r^{n}\|^2)+\nu\Delta t \|\nabla\bphi_r^{n+1}\|^2 
\leq  C\nu\Delta t \|\nabla \bfeta^{n+1}\|^2}
+ \frac{C\Delta t}{\nu}\|\bfeta^{n+1}\|\|\nabla \bfeta^{n+1}\|\|\nabla \bu^{n+1}\|^2\nonumber \\
&& + \frac{C\Delta t}{\nu^3}\|\bphi_r^{n+1}\|^2\|\nabla\bu^{n+1}\|^4
+\frac{C\Delta t}{\nu}\|\bw_r^{n+1}\|\|\nabla \bw_r^{n+1}\|\|\nabla\bfeta^{n+1}\|^2 
 +\frac{C\Delta t}{\nu}\|p^{n+1}-q_h\|^2 \nonumber
\\
&&+\frac{C\Delta t}{\nu}\| \bu_t^{n+1}-\frac{\bu^{n+1}-\bu^{n}}{\Delta t} \|^2.
\label{boundedterm}
\end{eqnarray}

We next get a bound for $\|\bphi_r^{n+1}\|^2$.  Write (\ref{step2}) by adding and subtracting the true solution projection ${\bf\mathcal{U}}^{n+1}$ on both sides to get
\begin{equation}
\bigg (\frac{\bphi_r^{n+1}-\btheta_r^{n+1}}{\Delta t},\bpsi\bigg ) =(\nu_T(I-P_{R})\nabla \frac{(\bphi_r^{n+1}+\btheta_r^{n+1}+2{\bf\mathcal{U}}^{n+1})}{2},(I-P_{R})\nabla \bpsi), \label{nw10}
\end{equation}
and then choose $\bpsi=\frac{(\bphi_r^{n+1}+\btheta_r^{n+1})}{2}$ in (\ref{nw10}) to obtain 
%and multiply both sides with $2\Delta t$, it follows that
\begin{eqnarray}
\|\bphi_r^{n+1}\|^2&=&\|\btheta_r^{n+1}\|^2+\frac{1}{2}\Delta t\nu_T\|(I-P_{R})\nabla (\bphi_r^{n+1}+\btheta_r^{n+1})\|^2 \nonumber
\\
&&+\Delta t(\nu_T(I-P_{R})\nabla {\bf\mathcal{U}}^{n+1},(I-P_{R})\nabla (\bphi_r^{n+1}+\btheta_r^{n+1})). \label{rela}
\end{eqnarray}
Noting ${\bf\mathcal{U}}^{n+1}=\bu^{n+1}-\bfeta^{n+1}$ and inserting (\ref{rela}) into (\ref{boundedterm}) results into
\begin{eqnarray}
\lefteqn{\|\btheta_r^{n+1}\|^2-\|\btheta_r^{n}\|^2+\frac{1}{2}\Delta t\nu_T\|(I-P_{R})\nabla (\bphi_r^{n+1}+\btheta_r^{n+1})\|^2
+ \nu\Delta t \|\nabla\bphi_r^{n+1}\|^2 }\nonumber
\\ 
&\leq& 
C\nu\Delta t \|\nabla \bfeta^{n+1}\|^2 + \frac{C\Delta t}{\nu}\|\bfeta^{n+1}\|\|\nabla \bfeta^{n+1}\|\|\nabla \bu^{n+1}\|^2\nonumber
\\
&&+  \frac{C\Delta t}{\nu^3}\|\nabla\bu^{n+1}\|^4\Big[\|\btheta_r^{n+1}\|^2+\frac{1}{2}\Delta t\nu_T\|(I-P_{R})\nabla (\bphi_r^{n+1}+\btheta_r^{n+1})\|^2 \nonumber
\\
&&+\Delta t(\nu_T(I-P_{R})\nabla (\bu^{n+1}-\bfeta^{n+1}), (I-P_{R})\nabla (\bphi_r^{n+1}+\btheta_r^{n+1})) \Big] \nonumber
\\
&&+\Delta t(\nu_T (I-P_{R})\nabla(\bfeta^{n+1}-\bu^{n+1}), (I-P_{R})\nabla (\bphi_r^{n+1}+\btheta_r^{n+1})) \nonumber
\\
&&+\frac{C\Delta t}{\nu}\|\bw_r^{n+1}\|\|\nabla \bw_r^{n+1}\|\|\nabla\bfeta^{n+1}\|^2
+\frac{C\Delta t}{\nu}\|p^{n+1}-q_h\|^2\nonumber
\\&&+\frac{C\Delta t}{\nu}\| \bu_t^{n+1}-\frac{\bu^{n+1}-\bu^{n}}{\Delta t} \|^2 \label{s1}
\end{eqnarray}
Assume now $\Delta t \leq \dfrac{1}{8C}\Big[\frac{\|\nabla \bu\|^4}{\nu^3}\Big]^{-1}$. Then, we can bound the remaining right hand terms of (\ref{s1}) as follows :
\begin{eqnarray}
{\frac{C(\Delta t)^2}{\nu^3}\nu_T \|\nabla\bu^{n+1}\|^4 \|(I-P_{R})\nabla (\bphi_r^{n+1}+\btheta_r^{n+1})\|^2 }\nonumber
\leq 
\frac{1}{8}\Delta t \nu_T\|(I-P_{R})\nabla (\bphi_r^{n+1}+\btheta_r^{n+1})\|^2.  \label{b11}
\end{eqnarray}
Similarly, using the preceding bound and Young's inequality we get
\begin{eqnarray}
\lefteqn{\frac{C(\Delta t)^2}{\nu^3}\|\nabla\bu^{n+1}\|^4 |(\nu_T(I-P_{R})\nabla (\bu^{n+1}-\bfeta^{n+1}), (I-P_{R}) \nabla (\bphi_r^{n+1}+\btheta_r^{n+1}))|}  \nonumber
\\
&\leq& C \Delta t |(\nu_T(I-P_{R})\nabla (\bu^{n+1}-\bfeta^{n+1}),(I-P_{R})\nabla (\bphi_r^{n+1}+\btheta_r^{n+1}))|   \nonumber
\\
&\leq& C \Delta t\nu_T \|(I-P_{R})\nabla (\bu^{n+1}-\bfeta^{n+1})\|^2 +  \frac{1}{8}\Delta t \nu_T\|(I-P_{R})\nabla (\bphi_r^{n+1}+\btheta_r^{n+1})\|^2 ,\label{b12}
\end{eqnarray}
and
%With Young's inequality gives the following bound
\begin{eqnarray}
\lefteqn{\Delta t(\nu_T (I-P_{R})\nabla(\bfeta^{n+1}-\bu^{n+1}), (I-P_{R})\nabla (\bphi_r^{n+1}+\btheta_r^{n+1})) } \nonumber
\\
&\leq& C \Delta t\nu_T \|(I-P_{R})\nabla (\bfeta^{n+1}-\bu^{n+1})\|^2 + \frac{1}{8}\Delta t \nu_T\|(I-P_{R})\nabla (\bphi_r^{n+1}+\btheta_r^{n+1})\|^2. \label{b13}
\end{eqnarray}
Substitute the bounds (\ref{b11})-(\ref{b13}) into (\ref{s1}) and 
sum from $n=0$ to $M-1$.  This gives
\begin{eqnarray}
\lefteqn{\|\btheta_r^{M}\|^2+ \sum_{n=0}^{M-1}\bigg[\frac{1}{8}\Delta t\nu_T\|(I-P_{R})\nabla (\bphi_r^{n+1}+\btheta_r^{n+1})\|^2+ \nu\Delta t \|\nabla\bphi_r^{n+1}\|^2\bigg]} \nonumber
\\
&\leq &\|\btheta_r^{0}\|^2 +C \sum_{n=0}^{M-1}\bigg[\nu\Delta t \|\nabla \bfeta^{n+1}\|^2 
+ \frac{\Delta t}{\nu}\|\nabla \bu^{n+1}\|^2\|\nabla\bfeta^{n+1}\|^2 \nonumber
\\&&
+ \Delta t\nu_T \|(I-P_{R})\nabla (\bu^{n+1}-\bfeta^{n+1})\|^2+
\frac{\Delta t}{\nu}\|\bw_r^{n+1}\|\|\nabla \bw_r^{n+1}\|\|\nabla\bfeta^{n+1}\|^2\nonumber
\\
&&+\frac{\Delta t}{\nu} \|p^{n+1}-q_h\|^2 +\frac{\Delta t}{\nu}\| \bu_t^{n+1}-\frac{\bu^{n+1}-\bu^{n}}{\Delta t} \|^2
\bigg] + \frac{C\Delta t}{\nu^3}  \sum_{n=0}^{M-1} \|\nabla\bu^{n+1}\|^4\|\btheta_r^{n+1}\|^2 \label{ss10}
\end{eqnarray}
Note that since $\bu_r^0={\bf \mathcal{U}}^0$, the first right hand side term in (\ref{ss10}) vanishes.  The second right hand side term is majorized using Lemma \ref{errpro}, as
\begin{eqnarray}
{\nu\Delta t  \sum_{n=0}^{M-1} \|\nabla \bfeta^{n+1}\|^2 } \leq  C \nu \Big(\big(h^{2m}+\|S_r\|_2h^{2m+2}\big)|||\bu|||^2_{2,m+1}+\sum_{j=r+1}^{d}\|\bpsi_j\|_1^2\lambda_j\Big), \label{es3}
\end{eqnarray}
%as similarly the 
Third term is bounded by using (\ref{assm2}) as follows,
\begin{eqnarray}
\lefteqn{\frac{\Delta t}{\nu}\sum_{n=0}^{M-1} \|\nabla \bfeta^{n+1}\|^2 \|\nabla \bu^{n+1}\|^2} \nonumber
\\
&\leq&\frac{\Delta t}{\nu} \sum_{n=0}^{M-1} \|\nabla \bu^{n+1}\|^2\Big(\big(h^{2m}+\|S_r\|_2h^{2m+2}\big)|||\bu|||^2_{2,m+1}+\sum_{j=r+1}^{d}\|\bpsi_j\|_1^2\lambda_j\Big)\nonumber\\
&\leq&{\nu}^{-1}|||\nabla \bu|||_{2,0}^2\Big(\big(h^{2m}+\|S_r\|_2h^{2m+2}\big)|||\bu|||^2_{2,m+1}+\sum_{j=r+1}^{d}\|\bpsi_j\|_1^2\lambda_j\Big)
%&\leq& \frac{\Delta t}{\nu} \Big(\sum_{n=0}^{M-1} \|\bfeta^{n+1}\|^2\|\nabla \bfeta^{n+1}\|^2 \Big)^{\frac{1}{2}}\Big(\sum_{n=0}^{M-1} \|\nabla \bu^{n+1}\|^4\Big)^{\frac{1}{2}}\nonumber\\
%&\leq& \frac{(\Delta t)^2}{2\nu} \sum_{n=0}^{M-1} \|\bfeta^{n+1}\|^2\|\nabla \bfeta^{n+1}\|^2 +\frac{1}{2\nu}\sum_{n=0}^{M-1} \|\nabla \bu^{n+1}\|^4\nonumber
%\\
%&\leq& C\Big(h^{2m+2}\frac{1}{M}\sum_{n=1}^{M}\|\bu^n\|_{m+1}+\sum_{j=r+1}^{d}\lambda_j\Big)\Big(\big(h^{2m}+\|S_r\|_2h^{2m+2}\big)\frac{1}{M}\sum_{n=1}^{M}\|\bu^n\|_{m+1}^2+\sum_{j=r+1}^{d}\|\bpsi_j\|_1^2\lambda_j\Big)\nonumber\\
%&&+\frac{1}{2\nu}\sum_{n=0}^{M-1} \|\nabla \bu^{n+1}\|^4
\end{eqnarray}

To bound the fourth term, one can proceed as follows. Using $\|(I-P_R)\nabla \bfeta^{n+1}\|\leq \|\nabla \bfeta^{n+1}\|$ along with Lemma \ref{errpro} leads to 
\begin{eqnarray}
 \lefteqn{\Delta t\nu_T\sum_{n=0}^{M-1} \|(I-P_{R})\nabla (\bu^{n+1}-\bfeta^{n+1})\|^2 } \nonumber
 \\
 &\leq&  \Delta t\nu_T\sum_{n=0}^{M-1}\|(I-P_{R})\nabla \bu^{n+1}\|^2
  + \Delta t\nu_T\sum_{n=0}^{M-1} \|(I-P_{R})\nabla\bfeta^{n+1}\|^2) \nonumber
\\
&\leq &
 \Delta t\nu_T\sum_{n=0}^{M-1}(\|\nabla \bu^{n+1}-P_{R}\nabla \bu^{n+1} \|^2+\Delta t\nu_T\sum_{n=0}^{M-1} \norm{\nabla \bfeta^{n+1}}^2 \label{trupro}
  \end{eqnarray}
For the first term on the right-hand side of (\ref{trupro}), we use (\ref{l2}) to find that
\begin{eqnarray}   
 \Delta t\nu_T\sum_{n=0}^{M-1}\|\nabla \bu^{n+1}-P_{R}\nabla \bu^{n+1} \|^2 &\leq& C\Delta t \nu_T\frac{1}{M}\sum_{n=0}^{M-1} 
 \inf_{\bv_R\in \bfX^r}  \|\nabla \bu^{n+1}-\nabla \bv_R^{n+1} \|^2 \nonumber
 \\
 &\leq& C\frac{1}{M} \sum_{n=0}^{M-1}  \|\nabla \bu^{n+1}-\nabla {\bf\mathcal{U}}^{n+1}_R \|^2  \label{kkk2}
 \end{eqnarray} 
where ${\bf \mathcal{U}}_R $ is the large scale representation of the projection.
Now using Lemma \ref{errpro}, the final estimation for (\ref{trupro}) becomes
 \begin{eqnarray}
 \lefteqn{\Delta t\nu_T\sum_{n=0}^{M-1} \|(I-P_{R})\nabla (\bu^{n+1}-\bfeta^{n+1})\|^2}  \label{es5}
 \\
 &\leq& C  \nu_T \bigg(\Big( h^{2m}+(\|S_R\|_2+\|S_r\|_2)h^{2m+2}\Big)|||\bu|||^2_{2,m+1}
 +\sum_{j=R+1}^{d}\|\bpsi_j\|_1^2\lambda_j+\sum_{j=r+1}^{d}\|
  \bpsi_j\|_1^2\lambda_j\bigg).\nonumber 
\end{eqnarray}
Corollary \ref{staw}, the stability of $\bw_r$ and (\ref{assm2}) provide an estimation for the fifth term in the right-hand side of (\ref{ss10}):
\begin{eqnarray}
\lefteqn{\dfrac{\Delta t}{\nu}\sum_{n=0}^{M-1}{\|\bw_r^{n+1}\|\|\nabla \bw_r^{n+1}\|}\|\nabla\bfeta^{n+1}\|^2} \nonumber
\\
&\leq&C\dfrac{\Delta t}{\nu}\sum_{n=0}^{M-1}\|\nabla \bw_r^{n+1}\|^2\|\nabla\bfeta^{n+1}\|^2\nonumber\\
&\leq&  C {\nu}^{-2}(\|\bu_r^0\|^2+\nu^{-1}|||\bff|||^2_{2,-1})  \Big((h^{2m}+\|S_r\|_2h^{2m+2})|||\bu|||^2_{2,m+1}+\sum_{j=r+1}^{d}\|\bpsi_j\|_1^2\lambda_j\Big).\label{wrb}
\end{eqnarray}

%\begin{eqnarray}
%\lefteqn{\dfrac{\Delta t}{\nu}\sum_{n=0}^{M-1}{\|\bw_r^{n+1}\|\|\nabla \bw_r^{n+1}\|}\|\nabla\bfeta^{n+1}\|^2} 
%\\
%&\leq&\dfrac{\Delta t}{\nu}(\sum_{n=0}^{M-1}\|\bw_r^{n+1}\|^2\|\nabla \bw_r^{n+1}\|^2)^{1/2}(\sum_{n=0}^{M-1}\|\nabla\bfeta^{n+1}\|^4)^{1/2} \nonumber\\
%&\leq&\dfrac{1}{\nu}\sum_{n=0}^{M-1}\|\bw_r^{n+1}\|^2\|\nabla \bw_r^{n+1}\|^2+\dfrac{(\Delta t)^2}{\nu}\sum_{n=0}^{M-1}\|\nabla\bfeta^{n+1}\|^2\|\nabla\bfeta^{n+1}\|^2\nonumber\\
%&\leq& C {\nu}^{-2} (\Delta t)^{-1}(\|\bu_r^0\|^2+\nu^{-1}|||\bff|||^2_{2,-1})^2\nonumber \\
%&&+ {\nu}^{-1}\Big((h^{2m}+\|S_r\|_2h^{2m+2})\frac{1}{M}\sum_{n=1}^{M}\|\bu^n\|_{m+1}^2+\sum_{j=r+1}^{d}\|\bpsi_j\|_1^2\lambda_j\Big)^2
%\end{eqnarray}
%\begin{eqnarray}
%\lefteqn{\dfrac{\Delta t}{\nu}\sum_{n=0}^{M-1}{\|\bw_r^{n+1}\|\|\nabla \bw_r^{n+1}\|}\|\nabla\bfeta^{n+1}\|^2} 
%\\
%&\leq&{\nu}^{-\frac{3}{2}}({\Delta t})^{-\frac{1}{2}}(\|\bu_r^0\|^2+\nu^{-1}|||\bff|||^2_{2,-1})^{3/2}\Delta t\sum_{n=0}^{M-1}\|\nabla\bfeta^{n+1}\|^2\nonumber\\
%&\leq&  {\nu}^{-\frac{3}{2}}({\Delta t})^{-\frac{1}{2}}(\|\bu_r^0\|^2+\nu^{-1}|||\bff|||^2_{2,-1})^{3/2} \Big((h^{2m}+\|S_r\|_2h^{2m+2})\frac{1}{M}\sum_{n=1}^{M}\|\bu^n\|_{m+1}^2+\sum_{j=r+1}^{d}\|\bpsi_j\|_1^2\lambda_j\Big).\nonumber
%\end{eqnarray}
The estimation of sixth term in the right hand side of (\ref{ss10}) uses approximation property (\ref{ap2}) to find
 \begin{eqnarray}
\frac{\Delta t}{\nu}\sum_{n=0}^{M-1} \|p^{n+1}-q_h\|^2 \leq \frac{C}{\nu}  h^{2m}|||p|||_{2,m}^2. \label{es7}
\end{eqnarray}
Finally, for the last term in the right hand side of (\ref{ss10}), Taylor series expansion with remainder in integral form is used along with Cauchy Schwarz and the triangle inequality to obtain
\begin{eqnarray}
\frac{\Delta t}{\nu} \sum_{n=0}^{M-1}\| \bu_t^{n+1}-\frac{\bu^{n+1}-\bu^{n}}{\Delta t} \|^2 
%&=&\frac{1}{\Delta t \nu} \sum_{n=0}^{M-1}\|\int_{t^n}^{t^{n+1}}(t-t^n)\bu_{tt} dt \| ^2 \nonumber
%\\
%&\leq& \frac{1}{\Delta t \nu} \sum_{n=0}^{M-1}\Big( \int_{t^n}^{t^{n+1}}(t-t^n)^2 dt\Big)\int_{t^n}^{t^{n+1}}\|\bu_{tt}\|^2 dt \nonumber
\leq C\nu^{-1}(\Delta t)^2 \|\bu_{tt}\|_{L^2(0,T;H^1(\Omega))}^2. \label{es8}
\end{eqnarray}
Collecting all the bounds (\ref{es3})-(\ref{es8}) for (\ref{ss10}) yields
\begin{eqnarray}
\lefteqn{\|\btheta_r^{M}\|^2+ \sum_{n=0}^{M-1}\bigg[\frac{1}{8}\Delta t\nu_T\|(I-P_{R})\nabla (\bphi_r^{n+1}+\btheta_r^{n+1})\|^2+ \nu\Delta t \|\nabla\bphi_r^{n+1}\|^2\bigg] }\nonumber
\\
&\leq& C\bigg[
\nu \bigg(\Big(h^{2m}+\|S_r\|_2h^{2m+2}
\Big)|||\bu|||^2_{2,m+1}
+\sum_{j=r+1}^{d}\|\bpsi_j\|_1^2\lambda_j\bigg) \nonumber
\\
&&+{\nu}^{-1}|||\nabla \bu|||_{2,0}^2 \bigg(\Big(h^{2m}+\|S_r\|_2h^{2m+2}
\Big)|||\bu|||^2_{2,m+1}
+\sum_{j=r+1}^{d}\|\bpsi_j\|_1^2\lambda_j\bigg)\nonumber
\\
&&+ \nu_T \bigg(\Big(h^{2m}+(\|S_R\|_2+\|S_r\|_2)h^{2m+2}\Big)|||\bu|||^2_{2,m+1}\nonumber
 \\ &&+\sum_{j=R+1}^{d}\|\bpsi_j\|_1^2\lambda_j+\sum_{j=r+1}^{d}\|\bpsi_j\|_1^2\lambda_j\bigg)+ {\nu}^{-2} (\|\bu_r^0\|^2+\nu^{-1}|||\bff|||^2_{2,-1})  \nonumber
 \\
 &&\times\bigg(\Big(h^{2m}+\|S_r\|_2h^{2m+2}
 \Big)|||\bu|||^2_{2,m+1}
 +\sum_{j=r+1}^{d}\|\bpsi_j\|_1^2\lambda_j\bigg)
 \nonumber
\\
&&+{\nu}^{-1} h^{2m}|||p|||_{2,m}^2+\nu^{-1}(\Delta t)^2 \|\bu_{tt}\|_{L^2(0,T;H^1(\Omega))}^2
+\frac{C\Delta t}{\nu^3}  \sum_{n=0}^{M-1} |||\nabla\bu|||_{\infty,0}^4\|\btheta_r^{n+1}\|^2 \bigg] \nonumber
\end{eqnarray}
Again using the assumption that $\Delta t \leq \frac{1}{8C}\bigg[\nu^{-3}{|||\nabla \bu|||_{\infty,0}^4}\bigg]^{-1}$ allows us to apply the discrete Gronwall inequality, which yields
\begin{eqnarray}
\lefteqn{\|\btheta_r^{M}\|^2+ \sum_{n=0}^{M-1}\bigg[\frac{1}{2}\Delta t\nu_T\|(I-P_{R})\nabla (\bphi_r^{n+1}+\btheta_r^{n+1})/2\|^2+ \nu\Delta t \|\nabla\bphi_r^{n+1}\|^2\bigg] }\nonumber
\\
&\leq& C\bigg[
\nu \bigg(\Big(h^{2m}+\|S_r\|_2h^{2m+2}
\Big)|||\bu|||^2_{2,m+1}
+\sum_{j=r+1}^{d}\|\bpsi_j\|_1^2\lambda_j\bigg)\nonumber
\\
&&+{\nu}^{-1}|||\nabla \bu|||_{2,0}^2\bigg(\Big(h^{2m}+\|S_r\|_2h^{2m+2} \Big)|||\bu|||^2_{2,m+1}+\sum_{j=r+1}^{d}\|\bpsi_j\|_1^2\lambda_j\bigg)\nonumber
\\
&&+ \nu_T \bigg(\Big(h^{2m}+(\|S_R\|_2+\|S_r\|_2)h^{2m+2}\Big)|||\bu|||^2_{2,m+1}\nonumber
\\&&+\sum_{j=R+1}^{d}\|\bpsi_j\|_1^2\lambda_j+\sum_{j=r+1}^{d}\|\bpsi_j\|_1^2\lambda_j\bigg)
+ {\nu}^{-2} (\|\bu_r^0\|^2+\nu^{-1}|||\bff|||^2_{2,-1})  \nonumber
\\
&&\times \bigg(\Big(h^{2m}+\|S_r\|_2h^{2m+2} \Big)|||\bu|||^2_{2,m+1}
+\sum_{j=r+1}^{d}\|\bpsi_j\|_1^2\lambda_j\bigg)\nonumber
\\
&&+{\nu}^{-1} h^{2m}|||p|||_{2,m}^2+\nu^{-1}(\Delta t)^2 \|\bu_{tt}\|_{L^2(0,T;H^1(\Omega))}^2\bigg]. \nonumber \label{gronwappll}
\end{eqnarray}
Finally, the triangle inequality is applied to produce the stated result.
%To obtain an error estimate for the error $\bu^M-\bu_r^M$ we apply triangle inequality \begin{eqnarray}
%\lefteqn{\|\bu^M-\bu_r^{M}\|^2+ \sum_{n=0}^{M-1}\bigg[\frac{1}{2}\Delta t\nu_T\|(I-P_{R})\nabla (\bu^{n+1}-(\bu_r^{n+1}+\bw_r^{n+1})/2)\|^2}\nonumber
%\\&&+ \nu\Delta t \|\nabla (\bu^{n+1}-\bw_r^{n+1})\|^2\bigg]\nonumber
%\\
%&\leq& \|\bfeta^M\|^2+\|\btheta_r^M\|^2+ \frac{1}{2}\Delta t\nu_T \sum_{n=0}^{M-1} (\|(I-P_{R})\nabla \bfeta^{n+1}\|^2 \nonumber
%\\
%&&+\|(I-P_{R})\nabla (\bphi_r^{n+1}+\btheta_r^{n+1})/2\|^2)+\nu \Delta t \sum_{n=0}^{M-1}(\|\nabla\bfeta^{n+1}\|^2+\|\nabla\bphi_r^{n+1}\|^2)\nonumber
%\\
%&\leq& C\bigg[h^{2m+2} +\Delta t ^2+\sum_{j=r+1}^{d}\lambda_j+
%\nu (h^{2m}+\|S_r\|_2h^{2m+2}+(1+\|S_r\|_2)\Delta t^2\nonumber
%\\
%&&+\sum_{j=r+1}^{d}\|\bpsi_j\|_1^2\lambda_j)+\nu^{-1} |||\nabla u|||_{2,0}^2\Big(h^{2m}+\|S_r\|_2h^{2m+2}+(1+\|S_r\|_2)\Delta t^2\nonumber
%\\
%&&+\sum_{j=r+1}^{d}\|\bpsi_j\|_1^2\lambda_j\Big)+\nu_T \Big(h^{2m}+(\|S_R\|_2+\|S_r\|_2)h^{2m+2}+(1+\|S_R\|_2\nonumber
%\\&&+\|S_r\|_2)\Delta t^2+\sum_{j=R+1}^{d}\|\bpsi_j\|_1^2\lambda_j+\sum_{j=r+1}^{d}\|\bpsi_j\|_1^2\lambda_j\Big)+(\Delta t) ^{\frac{1}{2}} {\nu}^{-\frac{3}{2}} (\|\bu_r^0\|  \nonumber
%\\
%&&+\nu^{-1/2}|||\bff|||_{2,-1})^2\Big(h^{2m}\|S_r\|_2h^{2m+2}+(1+\|S_r\|_{2})\Delta t^2
%+\sum_{j=r+1}^{d}\|\bpsi_j\|_1^2\lambda_j\Big)
%\nonumber
%\\
%&&+{\nu}^{-1} h^{2m}|||p|||_{2,m}^2+\nu^{-1}(\Delta t)^2 \|\bu_{tt}\|_{L^2(0,T;H^1(\Omega))}^2\bigg]. \nonumber \label{gronwappl}
%\end{eqnarray}
\end{proof}

\subsection{Extension to second order time stepping}

We consider now an extension of Algorithm \ref{alg1} to BDF2 time stepping.  

\begin{Algorithm}\label{alg2} 
Let $\bff\in L^{2}(0,T;\bfH^{-1}(\Omega))$ and initial conditions $\bu_r^0$ and $\bu_r^{-1}$ be given in $\bfX^r$.  Then for n=0,1,2,...\\
% $\bu_r^0=\bw_r^0=\bw_r^1=0$, $\bu_r^{n-1},\bu_r^{n}$ 
{\bf Step 1.}
Calculate  $\bw_r^{n+1}\in \bfX^r$ satisfying $\forall {\bpsi\in \bfX^r} $,
\begin{eqnarray} \label{step1n}
\qquad\bigg(\frac{3\bw_r^{n+1}-4\bu_r^{n}+\bu_r^{n-1}}{2\Delta t},\bpsi \bigg)+b(\bw_r^{n+1},\bw_r^{n+1},\bpsi)+ \nu(\nabla \bw_r^{n+1},\nabla \bpsi)=(\bff^{n+1},\bpsi)
\end{eqnarray}
{\bf Step 2.} Post-process $\bw_r^{n+1}$ to obtain $\bu_r^{n+1}\in \bfX^r$ satisfying $\forall {\bpsi\in \bfX^r}$,
\begin{eqnarray}\label{step2n}
\qquad\bigg (\frac{\bw_r^{n+1}-\bu_r^{n+1}}{\Delta t},\bpsi\bigg ) =(\nu_T(I-P_{R})\nabla \frac{(\bw_r^{n+1}+\bu_r^{n+1})}{2} , (I-P_{R})\nabla \bpsi).
% \quad \forall {\bpsi\in \bfX^r}
\end{eqnarray}
\end{Algorithm}

We note the post-processing step is exactly the same as in the backward Euler case.  Also as in the case of the backward Euler method above, without Step 2, Algorithm \ref{alg2} reduces to the classical Galerkin POD formulation for the NSE, although now using BDF2 time stepping.

We now prove stability of Algorithm \ref{alg2}.  A convergence result can be obtained by combining the ideas of Algorithm \ref{alg1}'s convergence proof with the stability proof below.  Such a proof is thus long and technical, but produces the expected result (i.e. same convergence as Algorithm \ref{alg1}, but second order in $\Delta t$ instead of first order).  This expected second order temporal convergence is illustrated in the numerical experiments section below.

\begin{lemma}\label{Lem:stan} 
The post-processed VMS-POD approximation is stable for the eddy viscosity term $\nu_T< 4\nu$ in the following sense:
	\begin{eqnarray}
\lefteqn{\|\bu_r^{M+1} \|^2+\|2\bu_r^{M+1}-\bu_r^M\|^2+2  \nu_T \Delta t \bigg\|(I-P_{R})\nabla \frac{(\bw_r^{M+1}+\bu_r^{M+1})}{2}\bigg\|^2} 	\nonumber
\\
&&+ 2\nu\Delta t \|\nabla\bw_r^{M+1}\|^2+\sum_{n=1}^{M}\|\bw_r^{n+1}-2\bu_r^{n}+\bu_r^{n-1}\|^2+(4 \nu-\nu_T)\frac{\Delta t}{2} \sum_{n=1}^{M-1} \|\nabla \bw_r^{n+1}\|^2 \nonumber
\\
&&\leq \|\bu_r^1\|^2+\|2\bu_r^{1}+\bu_r^{0}\|^2	+\frac{\nu_T \Delta t}{2} \|\nabla \bu_r^{1}\|^2
	+2\nu^{-1}|||\bff|||^2_{2,-1}.  \nonumber
	\label{stab3}
	\end{eqnarray}
\end{lemma}
\begin{proof}
Note that if we let $\bpsi=\frac{(\bw_r^{n+1}+\bu_r^{n+1})}{2}$ in Step 2, the numerical dissipation induced from Step 2 is given by
\begin{eqnarray}
\|{\bw_r^{n+1}\|^2= \|\bu_r^{n+1}}\|^2+2 \nu_T\Delta t\bigg\|(I-P_{R})\nabla \frac{(\bw_r^{n+1}+\bu_r^{n+1})}{2}\bigg\|^2. \label{numdisn}
\end{eqnarray}

Letting $\bpsi=\bw^{n+1}_{r}$ in (\ref{step1n}) and using the identity $$a(3a-4b+c)=\dfrac{1}{2}\big((a^2-b^2)+(2a-b)^2-(2b-c)^2+(a-2b+c)^2\big)$$
yields	
\begin{eqnarray}
&&\frac{1}{4\Delta t}\|\bw_r^{n+1}\|^2-\frac{1}{4\Delta t}\|\bu_r^n\|^2+\frac{1}{ 4\Delta t}(\|2\bw_r^{n+1}-\bu_r^{n}\|^2-\|2\bu_r^n-\bu_r^{n-1}\|^2)\nonumber
\\
&&+\frac{1}{ 4\Delta t}\|\bw_r^{n+1}-2\bu_r^{n}+\bu_r^{n-1}\|^2+\nu \|\nabla \bw_r^{n+1}\|^2=(\bff^{n+1},\bw_r^{n+1}). \label{stab1n}
\end{eqnarray}
Substitute (\ref{numdisn}) in (\ref{stab1n}), multiply both sides by $4\Delta t$, add $\|2\bu_r^{n+1}-\bu_r^n\|^2$ to both sides, and then apply the Cauchy-Schwarz inequality to get
\begin{eqnarray}
&&\|\bu_r^{n+1} \|^2-\|\bu_r^n\|^2+2 \nu_T\Delta t\bigg\|\|(I-P_{R})\nabla \frac{(\bw_r^{n+1}+\bu_r^{n+1})}{2}\bigg\|^2\nonumber
\\
&&+	(\|2 \bw_r^{n+1}-\bu_r^{n+1}\|^2-\|2\bu_r^{n+1}-\bu_r^n\|^2)+ (\|2\bu_r^{n+1}-\bu_r^n\|^2-\|2\bu_r^n-\bu_r^{n-1}\|^2) \nonumber
\\[5pt]
&&	+\|\bw_r^{n+1}-2\bu_r^{n}+\bu_r^{n-1}\|^2+2\nu \Delta t  \|\nabla \bw_r^{n+1}\|^2 \nonumber
\\
&&\leq C \nu^{-1}\Delta t\|\bff^{n+1}\|_{-1}^2.
 \label{stab1ni}
\end{eqnarray}
We now consider the term $\|2\bw_r^{n+1}-\bu_r^{n}\|^2-\|2\bu_r^{n+1}-\bu_r^n\|^2$ in (\ref{stab1ni}) . By using the properties of $L^2$ inner product, the equality (\ref{numdisn}) and rearranging terms gives
\begin{eqnarray}
\lefteqn{\|2\bw_r^{n+1}-\bu_r^{n}\|^2-\|2\bu_r^{n+1}-\bu_r^n\|^2} \nonumber		
\\		&=&(2\bw_r^{n+1}-\bu_r^{n},2\bw_r^{n+1}-\bu_r^{n})-(2\bu_r^{n+1}-\bu_r^n,2\bu_r^{n+1}-\bu_r^n)\nonumber\\
		&=&8(\bw_r^{n+1}-\bu_r^{n+1},-\frac{\bu_r^{n}}{2})+4(\|\bw_r^{n+1}\|^2-\|\bu_r^{n+1}\|^2)\nonumber\\
	&=&8  \nu_T\Delta t\big((I-P_R)\nabla \dfrac{(\bw_r^{n+1}+\bu_r^{n+1})}{2},(I-P_R)\nabla (-\frac{\bu_r^{n}}{2})\big)\nonumber\\
		&&+ 8 \nu_T\Delta t \big((I-P_R)\nabla \dfrac{(\bw_r^{n+1}+\bu_r^{n+1})}{2},(I-P_R)\nabla \dfrac{(\bw_r^{n+1}+\bu_r^{n+1})}{2} \big)\nonumber\\
		&=& 8 \nu_T\Delta t ((I-P_R)\nabla \frac{(\bw_r^{n+1}+\bu_r^{n+1})}{2},(I-P_R)\nabla \frac{(\bw_r^{n+1}+\bu_r^{n+1}-\bu_r^n)}{2}\nonumber\\
		&=& 8  \nu_T\Delta t \|(I-P_R)\nabla \frac{(\bw_r^{n+1}+\bu_r^{n+1}-\bu_r^n)}{2}\|^2\nonumber\\
		&&-8  \nu_T \Delta t ((I-P_R)\nabla (-\frac{\bu_r^n}{2}),(I-P_R)\nabla \frac{(\bw_r^{n+1}+\bu_r^{n+1}-\bu_r^n)}{2}). \label{extrmm1}
		\end{eqnarray}
Inserting (\ref{extrmm1}) in (\ref{stab1ni}) and applying Cauchy-Schwarz and Young's inequalities gives
\begin{eqnarray}
\lefteqn{\|\bu_r^{n+1} \|^2-\|\bu_r^n\|^2+\|2\bu_r^{n+1}-\bu_r^n\|^2	-\|2\bu_r^n-\bu_r^{n-1}\|^2} \nonumber
\\
&&+2\nu_T \Delta t \bigg\|(I-P_{R})\nabla \frac{(\bw_r^{n+1}+\bu_r^{n+1})}{2}\bigg\|^2
+\|\bw_r^{n+1}-2\bu_r^{n}+\bu_r^{n-1}\|^2
+2 \nu \Delta t  \|\nabla \bw_r^{n+1}\|^2 \nonumber
\\
&\leq& C\nu^{-1}\Delta t \|\bff^{n+1}\|_{-1}^2 
+2 \nu_T \Delta t \|(I-P_R)\nabla (\frac{\bu_r^n}{2})\|^2. \nonumber
\end{eqnarray}
Adding and subtracting terms for the last term in the previous inequality with the use of $\|I-P_R\|\leq 1$, we get
	\begin{eqnarray}
\lefteqn{
\|\bu_r^{n+1} \|^2-\|\bu_r^n\|^2+\|2\bu_r^{n+1}-\bu_r^n\|^2	-\|2\bu_r^n-\bu_r^{n-1}\|^2} 
\nonumber
\\
&&+2  \nu_T \Delta t\bigg\|(I-P_{R})\nabla \frac{(\bw_r^{n+1}+\bu_r^{n+1})}{2}\bigg\|^2+\|\bw_r^{n+1}-2\bu_r^{n}+\bu_r^{n-1}\|^2
+2 \nu \Delta t  \|\nabla \bw_r^{n+1}\|^2\nonumber
\\
&\leq& C \nu^{-1}\Delta t\|\bff^{n+1}\|_{-1}^2 +2\nu_T \Delta t\bigg\|(I-P_R)\nabla (\frac{\bw_r^n+\bu_r^n}{2})\bigg\|^2+\frac{\nu_T\Delta t }{2}  \|(I-P_R)\nabla \bw_r^n\|^2 \nonumber
\\&\leq& C\nu^{-1}\Delta t\|\bff^{n+1}\|_{-1}^2 
+2 \nu_T \Delta t \bigg\|(I-P_R)\nabla (\frac{\bw_r^n+\bu_r^n}{2})\bigg\|^2 
+\frac{ \nu_T\Delta t }{2}  \|\nabla \bw_r^n\|^2. \nonumber  
\end{eqnarray}
Summing over the time step $i=1,\cdots,M$ gives 	
\begin{eqnarray}
\lefteqn{\|\bu_r^{M+1} \|^2+\|2\bu_r^{M+1}-\bu_r^M\|^2+2  \nu_T \Delta t\bigg\|(I-P_{R})\nabla \frac{(\bw_r^{M+1}+\bu_r^{M+1})}{2}\bigg\|^2} 	\nonumber
\\
&&+\sum_{n=1}^{M}\|\bw_r^{n+1}-2\bu_r^{n}+\bu_r^{n-1}\|^2+2\nu \Delta t\|\nabla \bw_r^{M+1}\|^2+(4 \nu-\nu_T) \frac{\Delta t}{2}\sum_{n=1}^{M-1}  \|\nabla \bw_r^{n+1}\|^2 \nonumber
\\
&\leq& \|\bu_r^1\|^2+\|2\bu_r^{1}+\bu_r^{0}\|^2	+\frac{ \nu_T 
\Delta t}{2}\|\nabla \bw_r^{1}\|^2+2  \nu_T \Delta t \bigg\|(I-P_{R})\nabla \frac{(\bw_r^{1}+\bu_r^{1})}{2}\bigg\|^2 \nonumber
\\
&&+ 2 \nu^{-1}|||\bff|||^2_{2,-1}  \nonumber
\label{stab1ni2}
\end{eqnarray}
With the assumption $\bw_r^0=\bw_r^1=0$ and $\|I-P_R\|<1$, we obtain the stated result.
\end{proof}

\section{Numerical Studies}

This section gives results for three numerical experiments.  In all cases we use Algorithm \ref{alg2}, i.e. the scheme with second order time stepping.  Our first test considers the predicted convergence rates of the previous section, with respect to varying $R$ and $\Delta t$.  For the second test, we compare accuracy of the proposed VMS-POD scheme compared with the usual Galerkin POD method (i.e. unstabilized POD, computed by eliminating the post-processing step of the VMS-POD) in 2D channel flow past a cylinder.  Finally we consider the VMS-POD for a 3D turbulent channel flow simulation.

\begin{figure}[h!]
\begin{center}
\includegraphics[width=0.7\textwidth,height=0.24\textwidth, trim=0 0 0 0, clip]{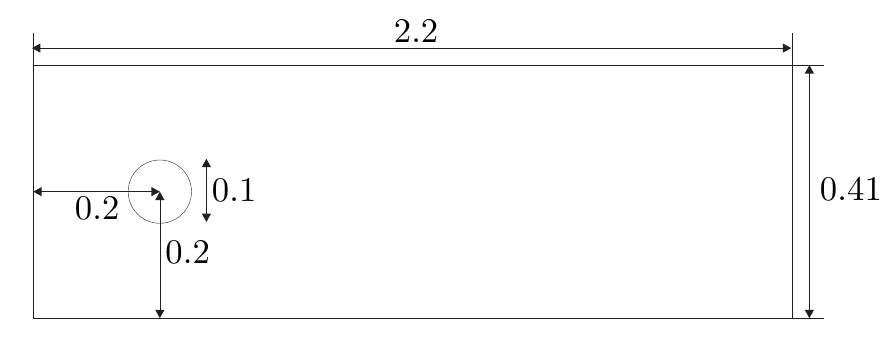}
\end{center}
\caption{\label{cyldomain} Shown above is the channel flow around a cylinder domain.}
\end{figure}

Our first two numerical tests use the test problem of 2D channel flow past a cylinder.  Here, the domain is a $2.2\times 0.41$ rectangle
 with a circle radius $=0.05$ centered at $(0.2,0.2)$ (see figure \ref{cyldomain}).  The test problem uses no slip boundary conditions for the walls and cylinder, and the time dependent inflow and outflow profiles are given by 
\begin{align*}
& u_{1}(0,y,t)=u_{1}(2.2,y,t)=\frac{6}{0.41^{2}}y(0.41-y) \, ,
\\
& u_{2}(0,y,t)=u_{2}(2.2,y,t)=0.
\end{align*}
The kinematic viscosity $\nu=10^{-3}$, and there is no forcing ($\bff={\bf 0}$).  The POD is created as described in section 2, by taking snapshots from a BDF2-finite element DNS simulation using Taylor-Hood elements, after a periodic-in-time solution is reached (a more detailed description of the setup is given in \cite{caiazzo2014numerical}).  The tests with the cylinder problem use the projection of the T=7 DNS solution as the initial condition, and run for $10$ time units.

\subsection{Numerical test 1: Convergence in $R$ and $\Delta t$}

We now test the predicted convergence rates of the previous section.  This is a particularly difficult task for this problem, as the parameters are the spatial mesh width $h$, time step $\Delta t$, POD cutoff $r$, and VMS cutoff $R$.  We assume the spatial mesh width is sufficiently small so that this error source is negligible.  Our particular interest here is the scaling of the error with the VMS cutoff $R$ and with the time step size $\Delta t$.  Due to having several parameters, to see convergence with respect to a particular parameter, it must be part of the dominant error source.  Hence in our tests, different parameter choices are made to see the various scalings.

To test the scaling of the error with $R$, we fix $r=12$, $\Delta t=0.002$, $\nu_T=0.0003$, and compute with varying $R$.  The error estimates are depend on $R$ via the term $\varepsilon=\sqrt{\sum\limits_{j=R+1}^{d}\|\varphi_j\|_1^2 \lambda_j}$, as well as $\| S_R \|_2$, and our interest
is the scaling with $\varepsilon$.  We note that for sufficiently large $R$, the term $\| S_R \|_2$ can become sufficiently large so that it becomes a dominant error source.  Results for this test are shown in table \ref{Rconv}, and we observe the expected convergence rate of approximately $1/2$ in the $L^2(0,T;H^1(\Omega))$ norm, and seemingly a higher rate in $L^{\infty}(0,T;L^2(\Omega))$ .  However, for larger $R$, we see a deterioration of the rate, and by $R=11$, the error increases (which is expected, due to increase in $\| S_R \|_2$ with $R$).  

To test the scaling with respect to $\Delta t$, we fix $r=16$, $R=14$, and vary $\Delta t$.  Errors and rates are given in table \ref{tconv}, and we observe rates consistent with second order.

\begin{table}[h!] 
	\begin{center}
		\begin{tabular}{ |c|c|c|c|cc|cc|}
			\hline
			$r$& $R$  & $\varepsilon$ & $\Delta t$ &  $ \|\bu-\bu_r\|_{L^{\infty}(L^2)}$ & rate & $\|\nabla (\bu-\bu_r)\|_{L^{2}(H^1)}$& rate \\ \hline
			12&3 &  18.1675 & 0.002 & 0.0294& - & 1.4060 & - \\ 
			12&5 &  10.4747 & 0.002 &0.0158& 1.13  & 1.0639& 0.51 \\ 
			12&7 & 4.7609&  0.002 &0.0076& 0.93  & 0.7039& 0.52\\ 
			12&9 & 2.6534 &  0.002 &0.0034& 1.37 & 0.6131& 0.24\\ 
			12&11 & 1.2920 & 0.002 &0.0090& -1.35  & 0.7580& -0.29\\ 
			\hline
		\end{tabular}
		\caption{\label{Rconv} Convergence of the VMS-POD for varying $R$.}
	\end{center}
\end{table}

\begin{table}[h!] 
	\begin{center}
		\begin{tabular}{ |c|c|c|cc|cc|}
			\hline
			$r$& $R$  &  $\Delta t$ &  $ \|\bu-\bu_r\|_{L^{\infty}(L^2)}$ & rate & $\|\nabla (\bu-\bu_r)\|_{L^{2}(H^1)}$& rate \\ \hline
			16&14&	0.032  &0.4343&-&15.6394&-\\ 
			16&14&	0.016  &0.3587&0.2759&10.8125&0.5325\\
			16&14&	0.008 &0.1296&1.4687&4.5266&1.2562\\ 
			16&14&	0.004 &0.0267&2.2792&1.0657&2.0866\\ 
%			16&14&	0.002& 0.0008& 5.0607& 0.1986 &2.4239\\ 
%
%			16& 5 & 0.032  &0.4344&-&15.6393&-\\ 
%			16&	5 & 0.016  &0.3568&0.2839&10.7443&0.5416\\
%			16&	5 & 0.008 &0.1144&1.6410&4.0505&1.4073\\ 
%			16&	5 & 0.004 &0.0259&2.1430&1.2844&1.6570\\ 
%			16&	5 & 0.002& 0.0164& 0.6592& 0.9935 &0.3705\\ 
			\hline
		\end{tabular}
		\caption{\label{tconv} Convergence of the VMS-POD for varying $\Delta t$.}
	\end{center}
\end{table}

\subsection{Numerical test 2: Error comparison of VMS-POD versus POD-G for 2D channel flow past a cylinder}

We now consider error comparison of the VMS-POD again the standard POD-G method.  The statistics of interest are the maximal drag and the maximal lift coefficients at the cylinder. The reference intervals are given for $c_{dmax}$ and $c_{lmax}$ in \cite{ST96},
\begin{eqnarray}\label{ref}
c_{dmax}^{ref} \in [3.22,3.24], \quad c_{lmax}^{ref} \in [0.98,1.02]
\end{eqnarray}
and recent computations of \cite{caiazzo2014numerical}, as well as our DNS that created the snapshots, are in agreement with these numbers.

We compute with $r=8$ modes, and find that the POD-G method is not good, and in particular we observe in figure \ref{fig:energy} that the energy is growing (seemingly) linearly with time.  But $T=10$, a significant increase in energy has occurred, leading to very inaccurate lift and drag predictions as well.

\begin{figure}[h!]
\centering
\includegraphics[width=.9\linewidth, height=0.17\textheight]{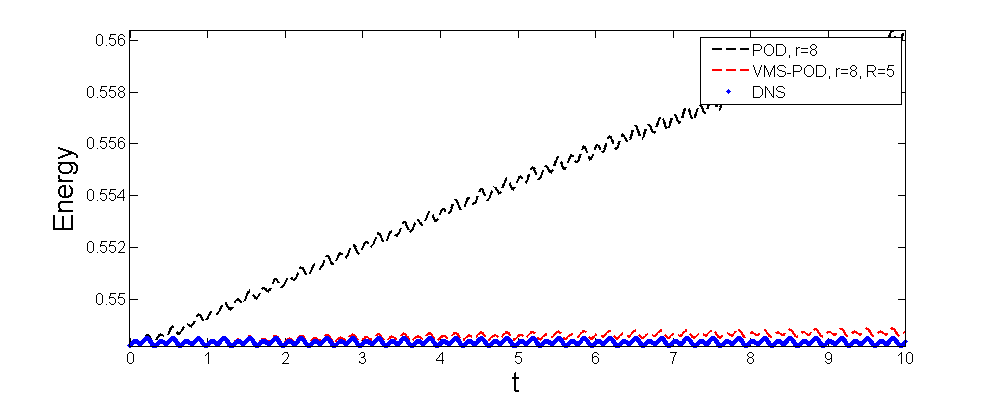}
\includegraphics[width=.9\linewidth, height=0.17\textheight]{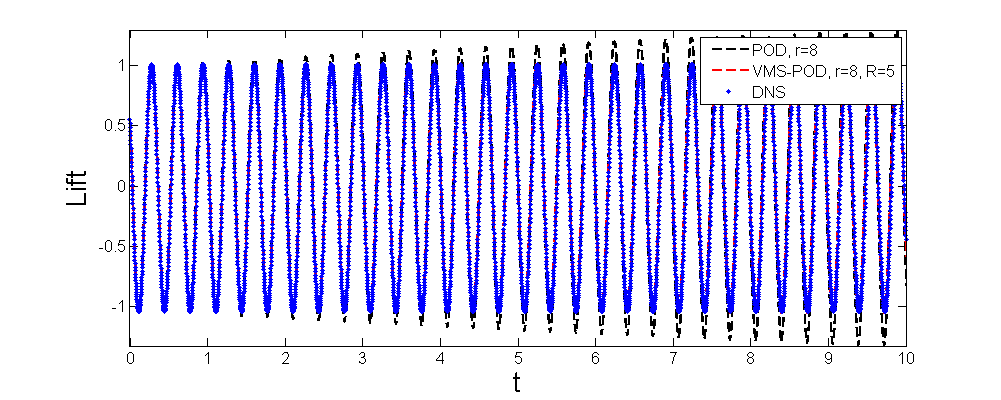}
\includegraphics[width=.9\linewidth, height=0.17\textheight]{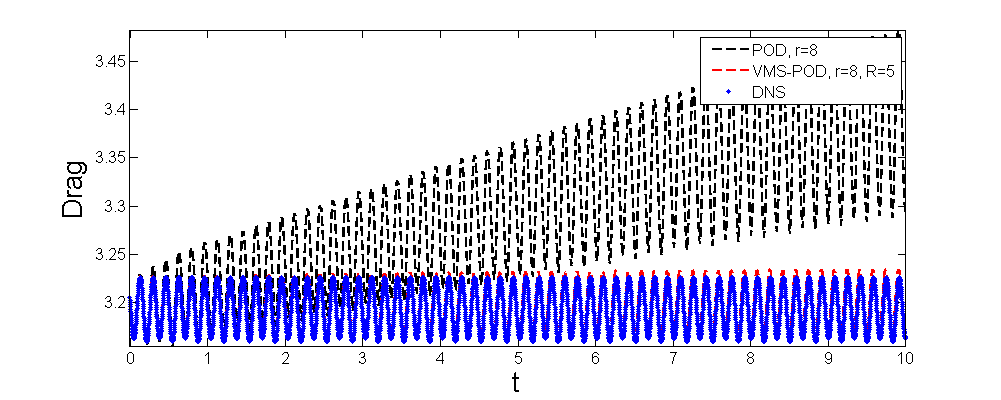}
\caption{Energy, lift and drag for DNS, POD-G and VMS-POD.}
\label{fig:energy}
\end{figure}

The VMS-POD does a much better job with prediction of energy, lift and drag, however.  Shown in figure \ref{fig:energy} is the energy, lift and drag with $R=5$ (and $\nu_T=0.0003$ was found to be optimal to within $\pm$ $0.000005$, in terms of matching the DNS energy at $T=10$). We note that with POD, such a parameter optimization is quite cheap.  Hence this is an excellent example of how the post-processing step of VMS-POD can remove the nonphysical energy growth that occurs with POD-G, and provide good reduced order solutions.

\subsection{Numerical test 3: VMS-POD for $Re_{\tau}$=395 turbulent channel flow}

For our last test, we consider the proposed VMS-POD for a benchmark turbulent channel flow problem at $Re_{\tau}=395$ \cite{MKM99,JR07}.
The domain $\Omega = (-2\pi,2\pi)\times (0,2) \times \left(-2\pi/3,2\pi/3\right)$, and for boundary conditions, no slip is enforced on the top and bottom walls $y=2$ and $y=0$, and periodic boundary conditions are enforced the remaining sides.  The kinematic viscosity is taken to be $\nu=\frac{1}{Re_{\tau}}$, and the flow is forced in the $x$-direction via ${\bf f}=\langle 1,0,0 \rangle$, and is also dynamically adjusted at each time step as in \cite{JR07} to maintain the bulk velocity.  We compute snapshots using the rNS-$\alpha$ model and scheme from \cite{RKB17,RZZ17}, by saving snapshots at each time step ($\Delta t=0.002$) of the simulation from $T=60$ to $T=70$, and then create the POD basis as described in section 2.  Plots of some of the POD basis functions are displayed in figure \ref{tcf1}.  We construct the reduced order model using $r=10$ modes, and will test the ability of the VMS-POD to give stable solutions.  We note that even though the snapshots were created with a turbulence model, we use the VMS-POD in the same form as above (i.e. step 1 is NSE, not the turbulence model).  

\begin{figure}[h!]
\centering
$\bpsi_1$ \hspace{2.3in} $\bpsi_2$\\
\includegraphics[width=.45\linewidth, height=0.3\textheight]{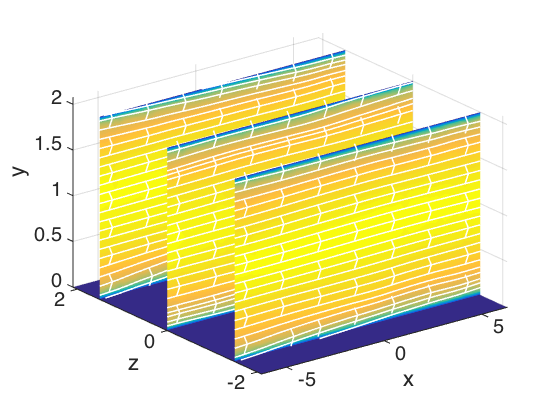}
\includegraphics[width=.45\linewidth, height=0.3\textheight]{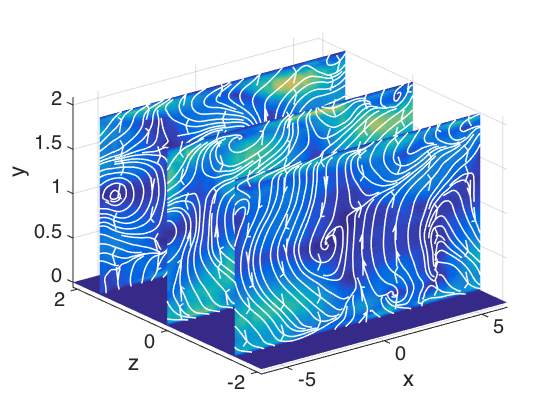}\\
$\bpsi_5$ \hspace{2.3in} $\bpsi_{10}$\\
\includegraphics[width=.45\linewidth, height=0.3\textheight]{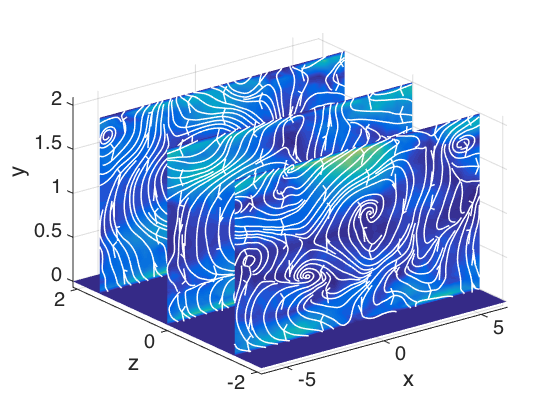}
\includegraphics[width=.45\linewidth, height=0.3\textheight]{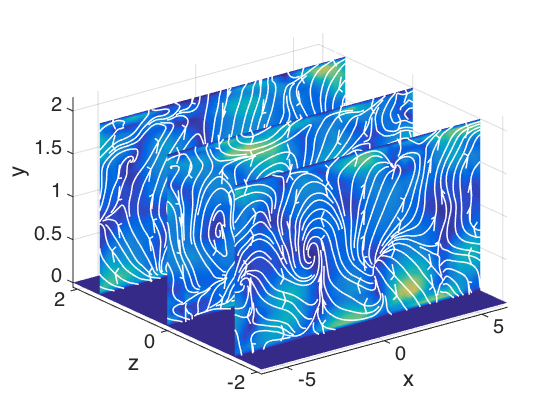}
\caption{\label{tcf1} Shown above are slices of velocity streamlines for the first, second, fifth and tenth modes from the POD for turbulent channel flow with rNS-$\alpha$ at $Re_{\tau}=395$.}
\end{figure}

Turbulent channel flow is well known to be a challenging numerical problem which causes many codes to fail and/or give poor solutions.  Underresolved simulations suffer from a nonphysical accumulation of energy near the smallest resolved scales.  In time, this accumulation of energy creates spurious oscillations, inaccuracies, and instability, often causing simulations to fail \cite{LR12,BIL06}.  Thus, we test here the ability of the VMS-POD to produce a stable solution.  We use the projections of the $T=60$ and $T=60.002$ turbulence model solutions as the initial conditions, and run the reduced order models for 5 time units.

We first test the usual POD-G, without any stabilization (i.e. $R=r=10$).  Energy versus time is shown for this simulation in figure \ref{tcf0}, and we observe that there is an energy blowup.  The instability occurs around $t=1$, and then the energy grows exponentially.

\begin{figure}[h!]
\centering
\includegraphics[width=.49\linewidth, height=0.22\textheight]{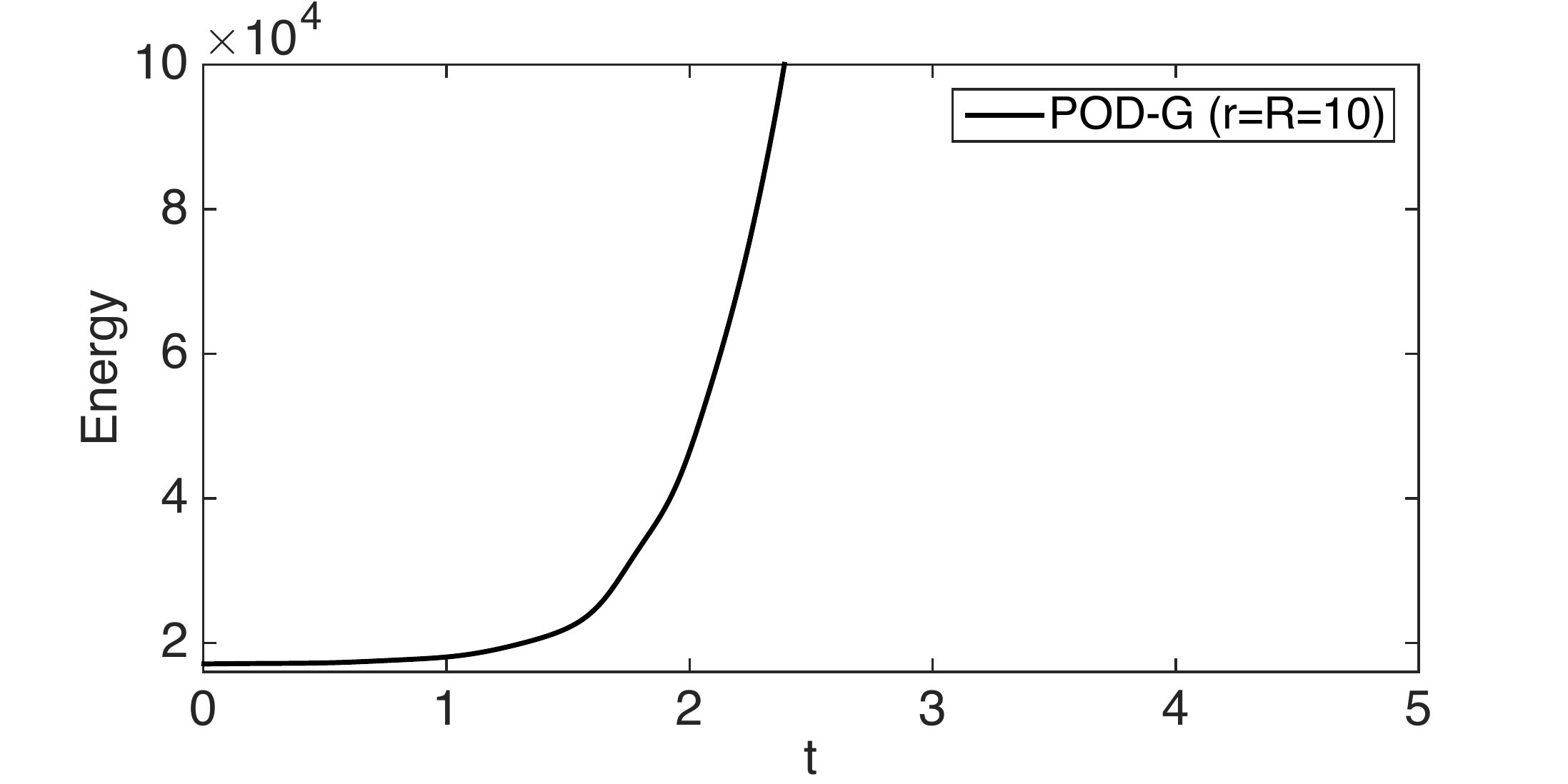}
\caption{\label{tcf0} Energy versus time (left) for the POD-G simulation with $R=10$.}
\end{figure}

We now test the VMS-POD with the $r=10$, and using $R=3$, $2$, and $1$ ($R=3$ is the largest $R$ for which we could get stable energy up to $t=5$).  Plots for energy versus time and mean streamwise velocity (average taken in both periodic directions and in time) are shown in figures \ref{tcf2}-\ref{tcf4} for these choices of $R$, and for each $R$ we test with several values of $\nu_T$.  The main observation we make from these results is that with sufficient stabilization using $R$ and $\nu_T$, the VMS-POD is able to produce energy stable solutions whose mean streamwise velocity profile matches the DNS values of Moser et al \cite{MKM99} (we note that the first mode $\bpsi_1$ also matches this mean streamwise profile as well).  For smaller values of $R$, smaller values of $\nu_T$ can be sufficient to obtain energy stable solutions.

\begin{figure}[h!]
\centering
\includegraphics[width=.49\linewidth, height=0.22\textheight]{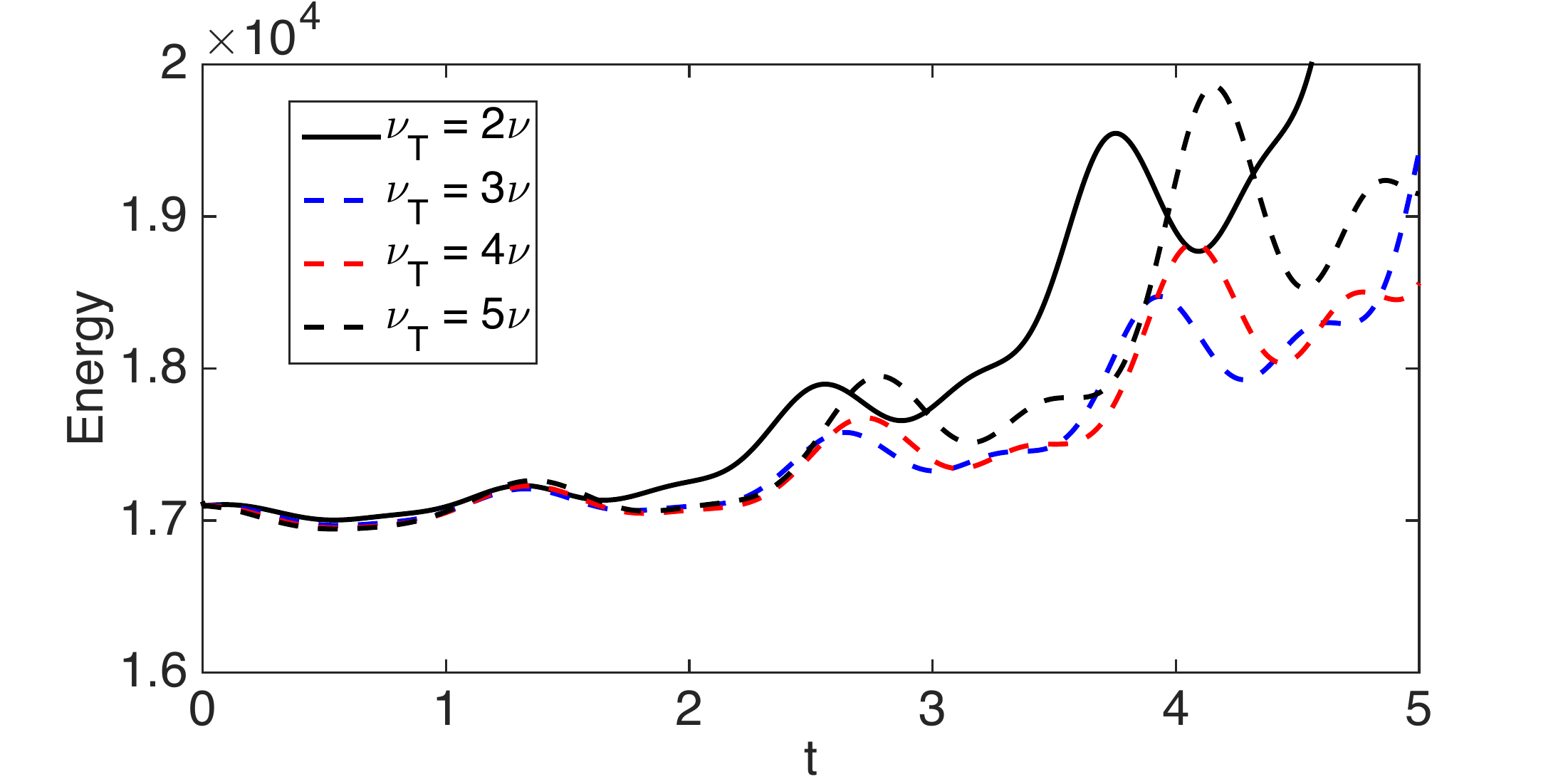}
\includegraphics[width=.49\linewidth, height=0.22\textheight]{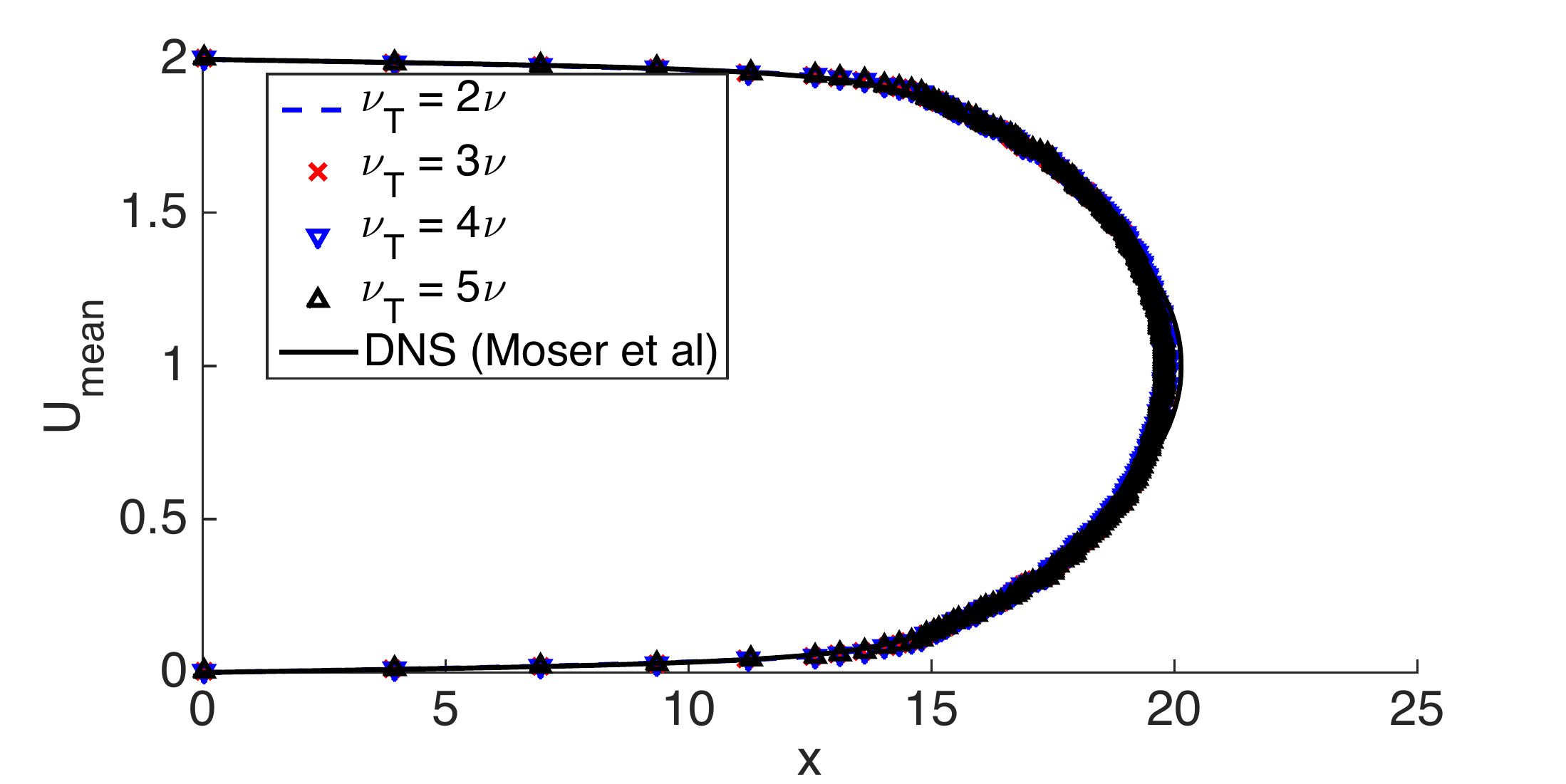}
\caption{\label{tcf2} Energy versus time (left) and mean streamwise velocity (right) with $r=10$, $R=3$ and varying $\nu_T$.}
\end{figure}

\begin{figure}[h!]
\centering
\includegraphics[width=.49\linewidth, height=0.22\textheight]{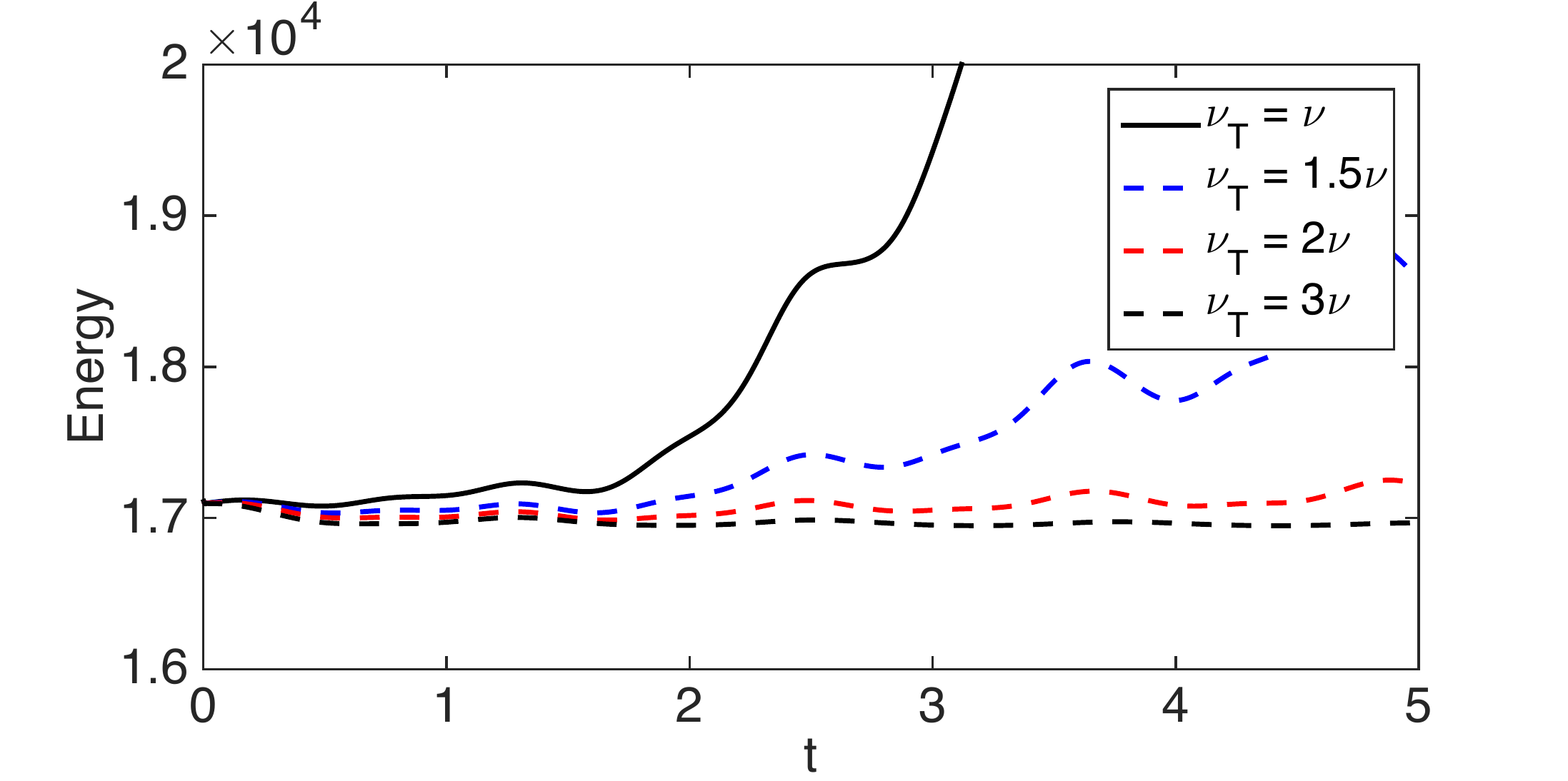}
\includegraphics[width=.49\linewidth, height=0.22\textheight]{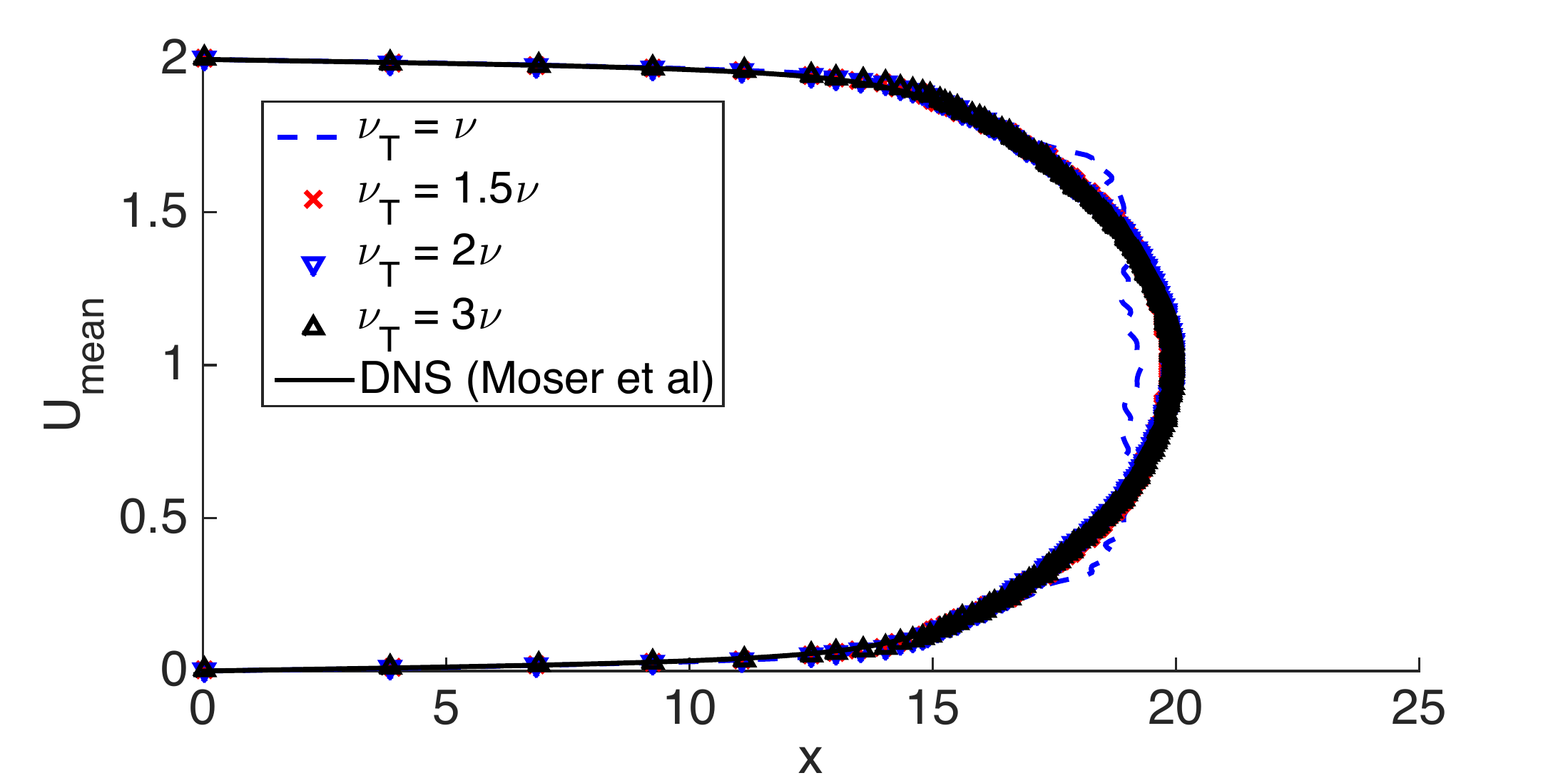}
\caption{\label{tcf3} Energy versus time (left) and mean streamwise velocity (right) with $r=10$, $R=2$ and varying $\nu_T$.}
\end{figure}

\begin{figure}[h!]
\centering
\includegraphics[width=.49\linewidth, height=0.22\textheight]{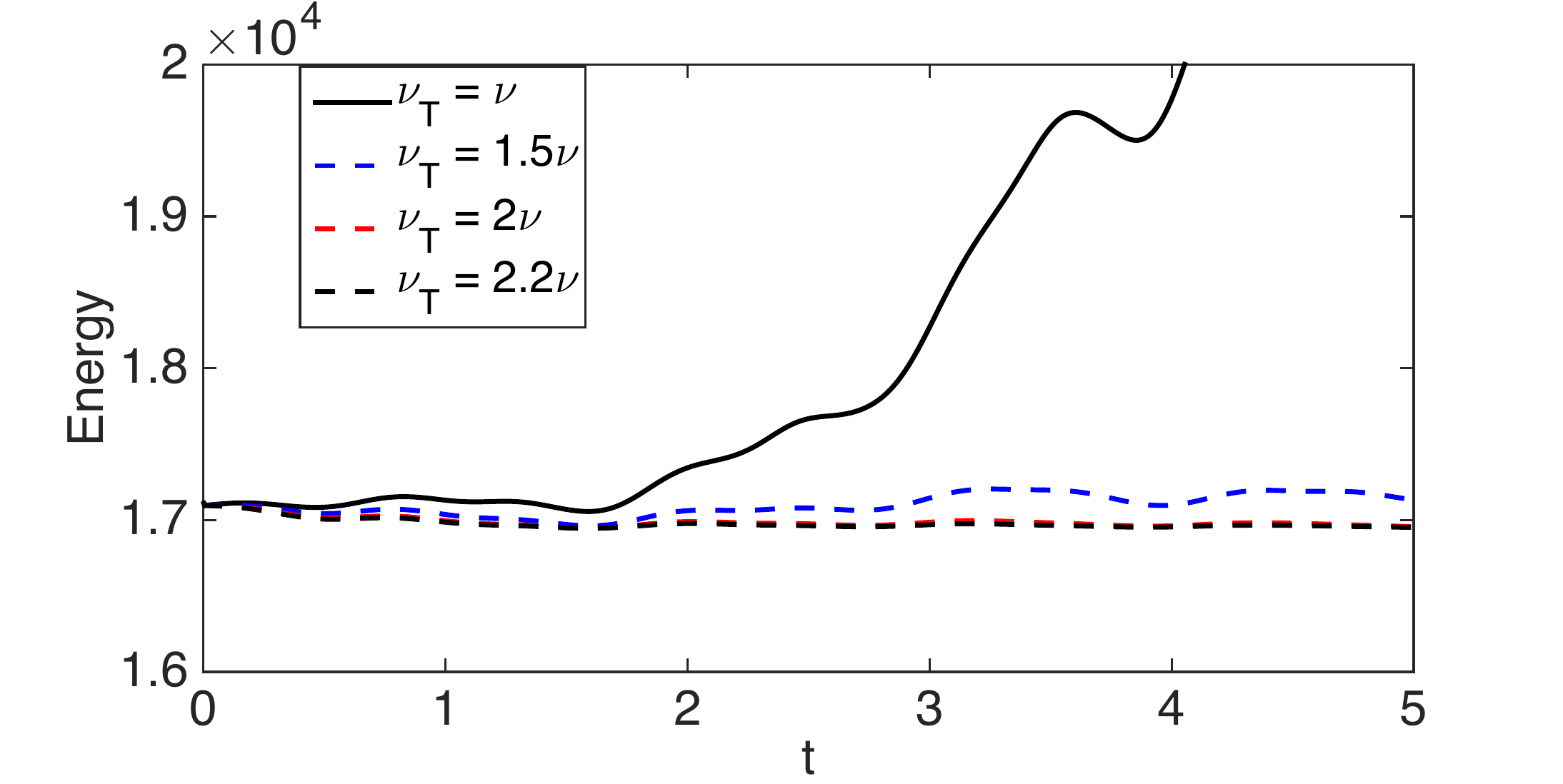}
\includegraphics[width=.49\linewidth, height=0.22\textheight]{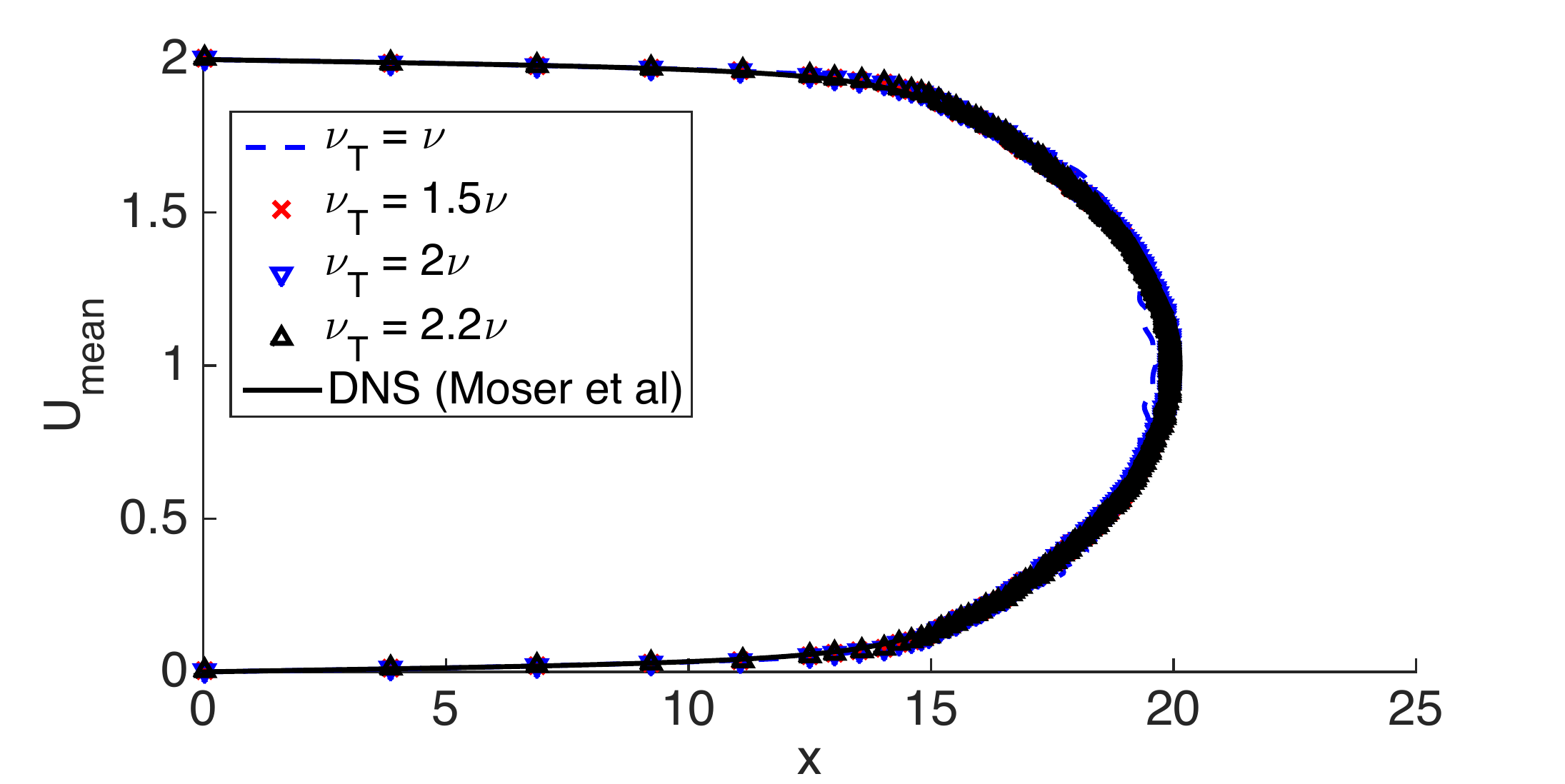}
\caption{\label{tcf4} Energy versus time (left) and mean streamwise velocity (right) with $r=10$, $R=1$ and varying $\nu_T$.}
\end{figure}

\section{Conclusion}
We proposed, analyzed and tested a VMS-POD method for incompressible NSE simulation, where the stabilization is completely decoupled into the second step of a two step implementation at each time step.  Decoupling of the stabilization has the advantage of easily being incorporated into existing POD-G codes, and also that stabilization parameters can be adjusted only in the stabilization step (and not as part of the evolution equation).  We rigorously prove an error estimate for the model, in terms of the number of POD modes $r$, the stabilization parameters $R$ (number of modes not to add stabilization to) and $\nu_T$, as well as the time step size $\Delta t$ and the mesh width $h$ of the underlying FEM simulation that produced the POD modes.

Results from several numerical experiments are provided that show how effective the method can be.  In particular, we show for 2D channel flow past a step, POD-G has an energy growth that causes poor lift and drag prediction, especially for longer times.  The proposed VMS-POD is able to fix this by stabilizing so that the energy matches the DNS energy, which in turn leads to excellent lift and drag prediction, even up to $t=10$ (and from the plots, it appear the accurate predictions can continue for even longer time).  Finally, we tested the ability of the VMS-POD to produce stable solutions for turbulent channel flow.  With $r=10$ modes, POD-G quickly becomes unstable, however the VMS-POD is able to produce simulations with stable energy and that yield accurate mean streamwise velocity profiles.

\bibliographystyle{amsplain}
\bibliography{reference}

\providecommand{\bysame}{\leavevmode\hbox to3em{\hrulefill}\thinspace}
\providecommand{\MR}{\relax\ifhmode\unskip\space\fi MR }
\providecommand{\MRhref}[2]{%
  \href{http://www.ams.org/mathscinet-getitem?mr=#1}{#2}
}
\providecommand{\href}[2]{#2}
\begin{thebibliography}{10}

\bibitem{ALT93}
N.~Aubry, W-Y. Lian, and E.S. Titi, \emph{Preserving symmetries in the proper
  orthogonal decomposition}, SISSC \textbf{14} (1993), 483--505.

\bibitem{B93}
G.~Berkooz, P.~Holmes, and J.~Lumley, \emph{The proper orthogonal decomposition
  in the analysis of turbulent flows}, Annual Review of Fluid Mechanics
  \textbf{25} (1993), 539--575.

\bibitem{BIL06}
L.~Berselli, T.~Illiescu, and W.~Layton, \emph{Mathematics of large eddy
  simulation of turbulent flows}, Springer, 2005.

\bibitem{caiazzo2014numerical}
A.~Caiazzo, T.~Iliescu, V.~John, and S.~Schyschlowa, \emph{A numerical
  investigation of velocity-pressure reduced order models for incompressible
  flows}, J. Comput. Phys. \textbf{259} (2014), 598--616.

\bibitem{C00}
A.~Chatterjee, \emph{An introduction to the proper orthogonal decomposition},
  Current Science \textbf{78} (2000), 808�817.

\bibitem{L03}
L.~Cordier and M.~Bergmann, \emph{Proper orthogonal decomposition: an
  overview}, Post Processing of Experimental and Numerical Data \textbf{Lecture
  Series 2003/2004} (2003), 1--45.

\bibitem{GR79}
V.~Girault and P.~A. Raviart, \emph{Finite element approximation of the
  {Navier}-{Stokes} equations}, Lecture Notes in Mathematics 749,
  Springer-Verlag, Berlin, 1979.

\bibitem{Gue99}
J.-L. Guermond, \emph{Stabilization of {G}alerkin approximations of transport
  equations by subgrid modeling}, M2AN \textbf{33} (1999), 1293 -- 1316.

\bibitem{Hu95}
T.J.R. Hughes, \emph{Multiscale phenomena: {G}reen's functions, the
  {D}irichlet-to-{N}eumann formulation, subgrid-scale models bubbles and the
  origin of stabilized methods}, Comp. Meth. Appl. Mech. Engrg. \textbf{127}
  (1995), 387 -- 401.

\bibitem{TZ12}
T.~Iliescu and Z.~Wang, \emph{Variational multiscale proper orthogonal
  decomposition: {Navier-Stokes} equations}, Numer. Meth. Partial. Diff. Eqs.
  (2012), 641--663.

\bibitem{TZ13}
\bysame, \emph{Variational multiscale proper orthogonal decomposition:
  convection-dominated convection-diffusion-reaction equations}, Mathematics of
  Computation \textbf{82(283)} (2013), 1357--1378.

\bibitem{jokaya1}
V.~John and S.~Kaya, \emph{A finite element variational multiscale method for
  the {Navier}-{Stokes} equations}, SIAM J. Sci. Comput. \textbf{26} (2005),
  1485--1503.

\bibitem{jokaya2}
\bysame, \emph{Finite element error analysis of a variational multiscale method
  for the {N}avier-{S}tokes equations}, Adv. Comput. Math. \textbf{28} (2008),
  43--61.

\bibitem{jokaya3}
V.~John, S.~Kaya, and W.~Layton, \emph{A two-level variational multiscale
  method for convection-diffusion equations}, Comput. Meth. Appl. Mech. Engrg.
  \textbf{195} (2005), 4594--4603.

\bibitem{JR07}
V.~John and M.~Roland, \emph{Simulations of the turbulent channel flow at
  ${R}e_{\tau}$=180 with projection-based finite element variational multiscale
  methods}, International Journal for Numerical Methods in Fluids \textbf{55}
  (2007), 407--429.

\bibitem{KV01}
K.~Kunisch and S.~Volkwein, \emph{Galerkin proper orthogonal decomposition
  methods for parabolic problems}, Numerische Mathematik \textbf{90(1)} (2001),
  117�148.

\bibitem{SWZ13}
W.~Layton L.~Shan and H.~Zheng, \emph{Numerical analysis of modular {VMS}
  methods with nonlinear eddy viscosity for the {N}avier-{S}tokes equations},
  International journal of numerical analysis and modeling \textbf{10} (2013),
  943--971.

\bibitem{Lay02}
W.~Layton, \emph{A connection between subgrid scale eddy viscosity and mixed
  methods}, Appl. Math. Comput. \textbf{133} (2002), 147 -- 157.

\bibitem{WJL8}
\bysame, \emph{Introduction to the numerical analysis of incompressible viscous
  flows}, Society for Industrial and Applied Mathematics (SIAM), Philadelphia,
  USA, 2008.

\bibitem{LR12}
W.~Layton and L.~Rebholz, \emph{Approximate deconvolution models of turbulence:
  Analysis, phenomenology and numerical analysis}, Springer-Verlag, 2012.

\bibitem{LRT11}
W.~Layton, L.~R\"ohe, and H.~Tran, \emph{Explicitly uncoupled {VMS}
  stabilization of fluid flow}, Comput. Methods Appl. Mech. Engrg. \textbf{200}
  (2011), 3183--3199.

\bibitem{MKM99}
R.~Moser, J.~Kim, and N.~Mansour, \emph{Direct numerical simulation of
  turbulent channel flow up to ${Re}_{\tau}$=590}, Physics of Fluids
  \textbf{11} (1999), 943--945.

\bibitem{H96}
J.~L.~Lumley P.~Holmes and G.~Berkooz, \emph{Turbulence, coherent structures,
  dynamical systems and symmetry}, Cambridge, 1996.

\bibitem{RKB17}
L.~Rebholz, T.-Y. Kim, and Young-Li Byon, \emph{On an accurate $\alpha$ model
  for coarse mesh turbulent channel flow simulation}, Applied Mathematical
  Modelling \textbf{43} (2017), 139--154.

\bibitem{RZZ17}
L.~Rebholz, C.~Zerfas, and K.~Zhao, \emph{Global in time analysis and
  sensitivity analysis for the reduced {NS}-$\alpha$ model of incompressible
  flow}, Journal of Mathematical Fluid Mechanics \textbf{to appear} (2017).

\bibitem{JPR13}
J.~P. Roop, \emph{A proper-orthogonal decomposition variational multiscale
  approximation method for a generalized {Oseen} problem}, Advances in
  Numerical Analysis \textbf{Volume 2013} (2013), 1--8.

\bibitem{ST96}
M.~Sch{\"a}fer and S.~Turek, \emph{Benchmark computations of laminar flow
  around a cylinder}, Wiesbaden: Vieweg (1996), 547--566.

\bibitem{L87a}
L.~Sirovich, \emph{Turbulence and the dynamics of coherent structures part
  {I}.coherent structures}, Q. Appl. Math \textbf{45} (1987), 561--571.

\bibitem{L87c}
\bysame, \emph{Turbulence and the dynamics of coherent structures part {III}.
  dynamics and scaling}, Q. Appl. Math \textbf{45} (1987), 583--590.

\bibitem{L87b}
\bysame, \emph{Turbulence and the dynamics of coherent structures part
  {II}.symmetries and transformations}, Q. Appl. Math \textbf{45} (1987),
  573--582.

\bibitem{XWWI17}
X.~Xie, Z.~Wang, D.~Wells, and T.~Iliescu, \emph{Numerical analysis of the
  {Leray} reduced order model}, submitted (2017).

\end{thebibliography}
\end{document}